\documentclass[12pt,a4paper]{article}
\usepackage{mathptmx}
\usepackage[T1]{fontenc}
\usepackage[a4paper]{geometry}
\usepackage{color}
\usepackage[pdftex]{graphicx}
\usepackage{xypic}
\usepackage{latexsym}
\usepackage{amsthm}
\usepackage{amsmath}
\usepackage{amssymb}

\begin{document}

\newcommand{\To}{\longrightarrow}
\newcommand{\h}{\mathcal{H}}
\newcommand{\s}{\mathcal{S}}
\newcommand{\A}{\mathcal{A}}
\newcommand{\K}{\mathcal{K}}
\newcommand{\B}{\mathcal{B}}
\newcommand{\W}{\mathcal{W}}
\newcommand{\M}{\mathcal{M}}
\newcommand{\Lom}{\mathcal{L}}
\newcommand{\T}{\mathcal{T}}
\newcommand{\F}{\mathcal{F}}

\newtheorem{definition}{Definition}[section]
\newtheorem{defn}[definition]{Definition}
\newtheorem{lem}[definition]{Lemma}
\newtheorem{prop}[definition]{Proposition}
\newtheorem{thm}[definition]{Theorem}
\newtheorem{cor}[definition]{Corollary}
\newtheorem{cors}[definition]{Corollaries}
\newtheorem{example}[definition]{Example}
\newtheorem{examples}[definition]{Examples}
\newtheorem{rems}[definition]{Remarks}
\newtheorem{rem}[definition]{Remark}
\newtheorem{notations}[definition]{Notations}
\theoremstyle{remark}
\theoremstyle{remark}
\theoremstyle{remark}

\theoremstyle{notations}
\theoremstyle{remark}
\theoremstyle{remark}
\theoremstyle{remark}
\newtheorem{dgram}[definition]{Diagram}
\theoremstyle{remark}
\newtheorem{fact}[definition]{Fact}
\theoremstyle{remark}
\newtheorem{illust}[definition]{Illustration}
\theoremstyle{remark}
\theoremstyle{definition}
\newtheorem{question}[definition]{Question}
\theoremstyle{definition}
\newtheorem{conj}[definition]{Conjecture}

\title{On $\mathcal{L}$-Injective Modules}

\author{Akeel Ramadan Mehdi\\ Department of Mathematics\\ College of Education\\Al-Qadisiyah University\\ Al-Qadisiyah, Iraq\\Email: akeel.mehdi@qu.edu.iq}

\maketitle

\abstract{Let $\mathcal{M}=\{(M,N,f,Q)\mid M,N,Q\in R\textup{-Mod}, \,N\leq M,\,f\in \textup{Hom}_{R}(N,Q)\}$ and let $\mathcal{L}$ be a nonempty subclass of  $\mathcal{M}.$ Jir\'{a}sko introduced the concepts of $\mathcal{L}$-injective
module as a  generalization of injective module as follows: a module $Q$ is said to be $\mathcal{L}$-injective if for each $(B,A,f,Q)\in \mathcal{L}$, there exists a homomorphism $g:B\rightarrow Q$ such that $g(a)=f(a),$ \, for all $a\in A$. The aim of this paper is to study $\mathcal{L}$-injective modules and some related concepts.}

$\vphantom{}$


$\vphantom{}$


\section{Introduction}

$\quad$  $\quad$\,
Throughout this article, unless otherwise specified, $R$  will denote an associative ring with non-zero identity, and all modules are left unital $R$-modules.   By a class of modules we mean a non-empty class of modules. The class of all left $R$-modules is denoted by  $R\textup{-Mod}$ and by  $\Re$ we mean the set $\{(M,N)\mid N\leq M,\, M\in R\textup{-Mod}\},$ where $N\leq M$ is a notation  means $N$ is a submodule of $M$. Given a family of modules $\{M_{i}\}_{i\in I},$ for each $j\in I,\, \pi_{j}:\bigoplus_{i\in I}M_{i}\rightarrow M_{j}$ denotes the canonical projection homomorphism. Let $M$ be a module and let $Y$  be a subset of $M$. The left annihilator of $Y$ in $R$ will be denoted by $l_{R}(Y),\;$ i.e., $l_{R}(Y)=\{r\in R\mid ry=0,\,\forall y\in Y\}.$ Given $a\in M,$ let $(Y:a)$  denote the set $\{r\in R\,\mid ra\in Y\},$ and let $ann_{R}(a):=(0:a).$ The right annihilator of a subset $I$ of $R$ in $M$ will be denoted by $r_{M}(I),$ i.e., $r_{M}(I)=\{m\in M\mid rm=0,\,\forall r\in I\}.$ The class $\{I\mid I \,\,\textup{is a left ideal of}\,\, R \,\,\textup{such that}\,\, ann_{R}(m)\subseteq I,\, \,\textup{for some}\, \,m\in M\}$ will be denoted by $\Omega(M).$

 An $R$-module $M$ is said to be injective if, for any module $B$,  every homomorphism $f:A\rightarrow M$, where $A$ is any submodule of $B,$ extends to a homomorphism $g:B\rightarrow M$ \cite{Bae40}. The notation $g\upharpoonright A= f$ means $g$ is an extension of $f$.
Let $M$ and $N$ be modules. Recall that  $N$ is said to be $M$-injective if every homomorphism from a submodule of $M$ to $N$ extends to a homomorphism from $M$ to $N$ \cite{AzMbVa75}. A module $M$  is said to be quasi-injective if $M$ is $M$-injective. The injective envelope of  a module $M$ will be denoted by $E(M)$.

Let $\tau=(\T,\F)$ be a  torsion theory. A submodule $B$ of a module $A$ is said to be $\tau$-dense in $A$ if $A/B$
is $\tau$-torsion (i.e. $A/B\in \T$). A submodule $A$ of a module $B$ is said to be $\tau$-essential in $B$ if it is $\tau$-dense and essential in $B$. A torsion theory $\tau$ is said to be noetherian if for every ascending chain $I_{1}\subseteq I_{2}\subseteq...$ of
left ideals of $R$ with $I_{\infty}=\bigcup_{j=1}^{\infty}I_{j}$ a $\tau$-dense left ideal in $R,$ there exists a positive integer
$n$ such that $I_{n}$ is $\tau$-dense in $R.$ A module $M$ is said to be $\tau$-injective if every homomorphism from a $\tau$-dense
submodule of $B$ to $M$ extends to a homomorphism from $B$ to $M$, where $B$ is any module \cite{Cri04}.
Let $M$ be an $R$-module. A $\tau$-injective envelope (or $\tau$-injective hull) of $M$ is a  $\tau$-injective module $E$ which is  a $\tau$-essential extension of $M$ \cite{Cha06}.  Every $R$-module $M$ has a $\tau$-injective envelope and it is unique up to isomorphism \cite{Cri04}. We use the notation $E_{\tau}(M)$ to stand for an $\tau$-injective envelope of $M$.
A $\tau$-injective module $E$ is said to be $\sum$-$\tau$-injective if $E^{(A)}$ is $\tau$-injective for any index set $A;$ \, $E$ is
said to be countably $\sum$-$\tau$-injective in case $E^{(C)}$ is $\tau$-injective for any countable index set $C.$ Let $E$ and  $M$ be
modules. Then  $E$ is said to be $\tau$-$M$-injective if any homomorphism from a $\tau$-dense submodule of $M$ to  $E$ extends to a homomorphism from $M$  to $E$. A module $E$ is said to be $\tau$-quasi-injective if $E$ is $\tau$-$E$-injective.

Let $\mathcal{M}=\{(M,N,f,Q)\mid M,N,Q\in R\textup{-Mod}, \,N\leq M,\,f\in \textup{Hom}_{R}(N,Q)\}$ and consider the following conditions on $\mathcal{L}$ that will be useful later, where $\mathcal{L}$  always denotes a nonempty subclass of $\M:$

$\vphantom{}$

$\left(\alpha\right)$ $(M,N,f,Q)\in\mathcal{L},\,(M,N^{\prime},f^{\prime},Q)\in\M$ and $(M,N,f,Q)\preceq(M,N^{\prime},f^{\prime},Q)$  implies
$(M,N^{\prime},f^{\prime},Q)\in\,\mathcal{L},$ where $\preceq$ is a partial order on  $\M$ defined by:\begin{center}
$(M,N,f,Q)\preceq(M^{\prime},N^{\prime},f^{\prime},Q^{\prime})\Longleftrightarrow M=M^{\prime},\,\,\,N\subseteq N^{\prime},\,\,\,Q=Q^{\prime},\,\,\,f^{\prime} \upharpoonright N= f.$
\par\end{center}

$\left(\beta\right)$ $(M,N,f,A)\in\mathcal{L},\,i:A\rightarrow B$ implies $(M,N,if,B)\in\mathcal{L},$   where $i$ is an inclusion homomorphism,

$\left(\gamma\right)$  $(M,N,f,A)\in\mathcal{L},\,g:A\rightarrow B$ an isomorphism, implies $(M,N,gf,B)\in\mathcal{L},$

$\left(\delta\right)$  $(M,N,f,A)\in\mathcal{L},\,g:A\rightarrow B$  a homomorphism, implies $(M,N,gf,B)\in\mathcal{L},$

$\left(\lambda\right)$ $(M,N,f,A)\in\mathcal{L},\, g:A\rightarrow B$ a split epimorphism, implies $(M,N,gf,B)\in\mathcal{L}$,

$\left(\mu\right)(M,N,f,Q)\in\mathcal{L}$, implies $(R,(N:x),f_{x},Q)\in\mathcal{L},\,\forall x\in M,$ where $f_{x}:(N:x)\rightarrow Q$ is a homomorphism define by $f_{x}(r)=f(rx)$, $\,\forall r\in(N:x)$,

$\vphantom{}$

Jir\'{a}sko in \cite{Jir75} introduced the concepts of $\mathcal{L}$-injective
module as a  generalization of injective module as follows: a module
$Q$ is said to be $\mathcal{L}$-injective if for each $(B,A,f,Q)\in \mathcal{L}$, there exists a homomorphism $g:B\rightarrow Q$ such that $(g \upharpoonright A)=f$. An $\mathcal{L}$-injective module $E$ is said to be an $\mathcal{L}$-injective envelope (or $\mathcal{L}$-injective hull) of a module $M$ if there is no proper $\mathcal{L}$-injective submodule of $E$ containing $M$ \cite{Jir75}. If a module $M$ has an $\mathcal{L}$-injective envelope and it is unique up to isomorphic then we will use the notation $E_{\mathcal{L}}(M)$ to stand for an $\mathcal{L}$-injective envelope of $M$. Clearly, injective module and all its generalizations are special cases of $\mathcal{L}$-injectivity.

The aim of this article is to study $\mathcal{L}$-injectivity  and some related concepts.

In section two, we give some   characterizations of $\mathcal{L}$-injective modules. For example, in Theorem~\ref{thm:(2.1.11)} we give a version of Baer's criterion for $\mathcal{L}$-injectivity.  Also, in Theorem~\ref{thm:(2.1.14)} we extend a characterization due to \cite[Theorem 2, p. 8]{Yi98} of $\mathcal{L}$-injective modules over commutative Noetherian rings.

In section three, we introduce the concepts of $\mathcal{L}$-$M$-injective module and $s$-$\mathcal{L}$-$M$-injective module as generalizations of $M$-injective modules and give some results about them.
For examples, in Theorem~\ref{thm:(2.2.5)} we prove that if $\mathcal{L}$ is a nonempty subclass of  $\mathcal{M}$ satisfies conditions $\left(\alpha\right)$, $\left(\beta\right)$ and $\left(\gamma\right)$ and $M,Q\in R\textup{-Mod}$ such that $M$ satisfies condition  $\left(E_{\mathcal{L}}\right)$, then   $Q$ is $\mathcal{L}$-$M$-injective if and only if   $f(M)\leq Q$,  for all $f\in\textup{Hom}_{R}(E_{\mathcal{L}}(M),E_{\mathcal{L}}(Q))$
with $(M,L,f$$\upharpoonright$$L,Q)\in\mathcal{L}$ where $L=\{m\in M\,\mid\:f(m)\in Q\}=M\bigcap f^{-1}(Q)$.
Also, in Proposition~\ref{prop:(2.2.14)} we generalize
\cite[Proposition 14.12, p. 66]{Cha06}, \cite[Proposition 1, p. 1954]{Bla90} and Fuchs's result in \cite{Fuc69}. Moreover,  our version of the Generalized Fuchs criterion is given in Proposition~\ref{prop:(2.2.15)} in which we prove that if $\mathcal{L}$ is a nonempty subclass of  $\mathcal{M}$ satisfies conditions
$\left(\alpha\right)$ and $\left(\mu\right)$ and $M,Q\in R\textup{-Mod}$ such that $M$ satisfies condition $\left(\mathcal{L}\right)$,
then a module $Q$ is $s$-$\mathcal{L}$-$M$-injective if and only if for each $(R,I,f,Q)\in\mathcal{L}$ with
$\textup{ker}(f)\in\Omega(M)$, \, there exists an element  \, $x\in Q$  such that  $f(a)=ax,  \;\forall a\in I$.

In section four, we study direct sums of $\mathcal{L}$-injective modules. In Proposition~\ref{prop:(2.3.3)} we prove that for any family $\{E_{\alpha}\}_{\alpha\in A}$ of  $\mathcal{L}$-injective modules , where $A$ is an infinite index set, if  $\mathcal{L}$ satisfies conditions $\left(\alpha\right),\,\,\,\left(\mu\right)$ and $\left(\delta\right)$ and $\bigoplus_{\alpha\in C}E_{\alpha}$ is an $\mathcal{L}$-injective module
for any countable subset $C$ of $A,$  then $\bigoplus_{\alpha\in A}E_{\alpha}$ is an $\mathcal{L}$-injective module. In Theorem~\ref{thm:(2.3.10)}, we prove that for any nonempty subclass $\mathcal{L}$ of $\mathcal{M}$ which satisfies conditions $(\alpha)$ and $(\delta)$ and for any nonempty class $\mathcal{K}$ of modules  closed under isomorphic copies and $\mathcal{L}$-injective hulls, if the direct sum of any family $\{E_{i}\}_{i\in N}$ of $\mathcal{L}$-injective $R$-modules in $\mathcal{K}$ is $\mathcal{L}$-injective, then every ascending chain $I_{1}\subseteq I_{2}\subseteq...$ of left ideals of $R$ in $H_{\mathcal{K}}(R)$ with $I_{\infty}=\bigcup_{j=1}^{\infty}I_{j}$\,\, $s$-$\mathcal{L}$-dense   in $R,$ terminates.  Also, in Theorem~\ref{thm:(2.3.12)} we generalize results in \cite[p. 643]{PaZh94} and  \cite[Proposition 5.3.5, p. 165]{Cri04} in which we prove that for any nonempty subclass $\mathcal{L}$ of $\mathcal{M}$ which satisfies conditions $\left(\alpha\right),\,\,\left(\mu\right),\,\,\left(\delta\right)$
and $\left(I\right)$  and for any nonempty class $\mathcal{K}$ of modules  closed under isomorphic copies and submodules, if every ascending chain
$J_{1}\subseteq J_{2}\subseteq...$ of left ideals of $R$ such that $(J_{i+1}/J_{i})\in \mathcal{K}$,  $\forall i\in\mathbb{N}$
and $J_{\infty}=\bigcup_{i=1}^{\infty}J_{i}$ \,\, $s$-$\mathcal{L}$-dense in $R$ terminates, then every direct sum of $\mathcal{L}$-injective modules in $\mathcal{K}$ is $\mathcal{L}$-injective.

Finally,  in section five,   we introduce the concept of \, $\sum$-$\mathcal{L}$-injectivity
as a generalization of $\sum$-injectivity and $\sum$-$\tau$-injectivity and prove  Theorem~\ref{thm:(2.4.4)} in which we generalize Faith's result  \cite[Proposition 3, p. 184]{Fai66} and \cite[Theorem 16.16, p. 98]{Cha06}.

\section{Some  Characterizations of  $\mathcal{L}$-Injective Modules}

$\quad$ \,One well-known result concerning injective modules states that an $R$-module $M$ is injective if and only if every homomorphism from a left
ideal of $R$ to $M$ extends to a homomorphism from $R$ to $M$ if and only if for each left ideal $I$ of $R$ and every $f\in \textup{Hom}_{R}(I,M)$, there is an $m\in M$ such that $f(r)=rm, \,\forall r\in I.$ This is known as Baer's condition \cite{Bae40}. Baer's result shows that the left ideals of $R$ form a test set for injectivity.

$\vphantom{}$

The following theorem   gives a version of Baer's criterion for $\mathcal{L}$-injectivity.

\begin{thm}\label{thm:(2.1.11)} \textbf{(Generalized Baer's Criterion)} Consider the following three conditions for an $R$-module $M$:

$\left(1\right)$ $M$ is $\mathcal{L}$-injective;

$\left(2\right)$ for every  $(R,I,f,M)\in\mathcal{L}$, there exists an $R$-homomorphism $g\in \textup{Hom}_{R}(R,M)$ such that $g(a)=f(a),$ for all $a\in I$;

$\left(3\right)$ for each $(R,I,f,M)\in\mathcal{L}$, there exists an element  $m\in M$ such that $f(r)=rm, \;\forall r\in I$.

Then $\left(2\right)$ and $\left(3\right)$ are equivalent and $\left(1\right)$ implies $\left(2\right)$. Moreover, if $\mathcal{L}$
satisfies conditions $\left(\alpha\right)$ and $\left(\mu\right)$, then all the three conditions are equivalent.
\end{thm}

\begin{proof}
$\left(1\right)\Rightarrow\left(2\right)$ and $\left(2\right)\Leftrightarrow\left(3\right)$ are obvious.

$\left(2\right)\Rightarrow\left(1\right)$ Let $\mathcal{L}$ satisfy conditions $\left(\alpha\right)$ and $\left(\mu\right)$ and let $(B,A,f,M)\in\mathcal{L}.$
Let  $S=\{(C,\varphi)\mid A\leq C\leq B\:,\:\varphi\in \textup{Hom}_{R}(C,M)$
such that $(\varphi$$\upharpoonright$$ A)=f\:\}.$ Define on $S$ a partial order $\preceq$
by
\[
(C_{1},\varphi_{1})\preceq(C_{2},\varphi_{2})\Longleftrightarrow C_{1}\leq C_{2}\; \textup{and}\;(\varphi_{2}\upharpoonright C_{1})=\varphi_{1}\]

Clearly, $S\neq\emptyset$ since $(A,f)\in S$. Furthermore, one can show that $S$ is inductive in the following manner. Let $F=\{(A_{i},f_{i})\mid i\in I\}$ be an ascending chain in $S$. Let $A_{\infty}=\cup_{i\in I}A_{i}$. Then for any $a\in A_{\infty}$ there is a $j\in I$ such that $a\in A_{j}$, and so we can define $f_{\infty}:A_{\infty}\rightarrow M$, by $f_{\infty}(a)=f_{j}(a)$. It is straightforward to check that $f_{\infty}$ is well defined and $(A_{\infty},f_{\infty})$ is an upper bound for $F$ in $S$. Then  by Zorn\textquoteright{}s Lemma,  $S$ has a maximal
element,  say  $(B^{\prime},g^{\prime})$.  We  will  prove  that  $B^{\prime}=B$.

Suppose that there exists $x\in B\setminus B^{\prime}$. It is clear that $(B,A,f,M)\preceq(B,B^{\prime},g^{\prime},M)$. Since $(B,A,f,M)$ $\in\mathcal{L}$ and $\mathcal{L}$ satisfies condition $\left(\alpha\right)$, it follows that $(B,B^{\prime},g^{\prime},M)\in\mathcal{L}.$ Since $\mathcal{L}$
satisfies condition $\left(\mu\right)$, we have  that $(R,(B^{\prime}:x),g_{x}^{\prime},M)\in\mathcal{L}.$ By hypothesis, there exists a homomorphism $g:R\rightarrow M$ such that $g(r)=g_{x}^{\prime}(r)=g^{\prime}(rx)$, \, $\forall r\in(B^{\prime}:x)$.
 Define $\psi:B^{\prime}+Rx\rightarrow M$ by $\psi(b+rx)=g^{\prime}(b)+g(r)$,\,
$\forall b\in B^{\prime}$,\, $\forall r\in R.$ It is clear that $\psi$ is a well-defined homomorphism and $(B^{\prime},g^{\prime})\preceq(B^{\prime}+Rx,\psi)$. Since $(B^{\prime}+Rx,\psi)\in S$ and $B^{\prime}\subsetneqq B^{\prime}+Rx$,  we have a contradiction to maximality of $(B^{\prime},g^{\prime})$
in $S$. Hence $B^{\prime}=B$ and this means that there exists a homomorphism $g^{\prime}:B\rightarrow M$ such that $(g^{\prime}$$\upharpoonright$$A)=f$. Thus $M$ is $\mathcal{L}$-injective.
\end{proof}

Now we will introduce the concept of $P$-filter as follows.

\begin{defn}\label{defn:(2.1.7)} Let $\Re=\{(M,N)\mid N\leq M,\, M\in R\textup{-Mod}\}$ and let $\rho$ be a nonempty subclass of $\Re$. We say that $\rho$ is a $P$-filter if $\rho$ satisfies the following conditions:

$\left(i\right)$ if $(M,N)\in\rho$ and $N\leq K\leq M$, then $(M,K)\in\rho$;

$\left(ii\right)$ for all  $M\in R\textup{-Mod}$, $(M,M)\in\rho$;

$\left(iii\right)$ if $(M,N)\in\rho$, then $(R,(N:x))\in\rho, \,\forall x\in M.$
\end{defn}

\begin{example}\label{example:(2.1.8)} All of the following subclasses of $\Re$ are $P$-filters.

$\left(1\right)$ $\rho_{\T}=\{(M,N)\in\Re\mid N\leq M$ such that $M/N\in \T, \,\,M\in R\textup{-Mod}\},$
where $\T$ is a nonempty class of modules closed under submodules and homomorphic images.

$\left(2\right)$ $\rho_{\infty}=\Re=\{(M,N)\mid N\leq M,\,\, M\in R\textup{-Mod}\}.$

$\left(3\right)$ $\rho_{\tau}=\{(M,N)\in\Re\mid N$ is  $\tau$-dense in $M,$ \,\,$M\in R\textup{-Mod}\},$ where $\tau$ is
a hereditary torsion theory.

$\left(4\right)$ $\rho_{r}=\{(M,N)\in\Re\;\mid\; N\leq M$ such that $r(M/N)=M/N,\,\,\, M\in R\textup{-Mod}\},$ where $r$ is a left exact preradical.

$\left(5\right)$ $\rho_{max}= \{(M,N)\in\Re\mid N$ is a maximal submodule in $M$ or $N=M,$$\, \,\,M\in R\textup{-Mod}\}.$

$\left(6\right)$ $\rho_{e}=\{(M,N)\in\Re\mid N\leq_{e}M,\,\,\, M\in R\textup{-Mod}\}.$

\end{example}

$\vphantom{}$

It is clear that the $P$-$\textup{filters}$ from $\left(2\right)$ to $\left(5\right)$ are special cases of $P$-filter in $\left(1\right)$. Also, if
$\rho$ is a $P$-filter then the subclass $\rho_{R}=\{(R,I)\in\rho\mid I$ is a left ideal of $R\}$ of $\Re$ is also $P$-filter.

\begin{notations}\label{notations:(2.1.9(a))} We will fix the following notations.

- For any two $P$-filters $\rho_{1}$ and $\rho_{2}$, we will denote by $\mathcal{L}_{(\rho_{1},\rho_{2})}$ the subclass $\mathcal{L}_{(\rho_{1},\rho_{2})}=\{(M,N,f,Q)\in\M\mid M,N,Q\in R\textup{-Mod},\,(M,N)\in\rho_{1}$ and $f\in \textup{Hom}_{R}(N,Q)$ such that $(M,\textup{ker}(f))\in\rho_{2}\}.$

-For any two nonempty classes of modules $\T$ and $\F,$ we will denote by $\mathcal{L}_{(\T,\F)}$ the subclass $\mathcal{L}_{(\T,\F)}=\{(M,N,f,Q)\in\M\mid M,N,Q\in R\textup{-Mod},\, N\leq M$ such that $M/N\in \T$ and $f\in \textup{Hom}_{R}(N,Q)$  with $M/\textup{ker}(f)\in \F\}$. It is clear that $\mathcal{L}_{(\T,\F)}=\mathcal{L}_{(\rho_{\T},\rho_{\F})},$  when $\T$ and $\F$ are closed under submodules and homomorphic images.

- For any two preradicals $r$ and $s$, we will denote by $\mathcal{L}_{(r,s)}$ the subclass $\mathcal{L}_{(r,s)}=\{(M,N,f,Q)\in\M\,\,\mid  \,\,M,N,Q\in R\textup{-Mod}, \,\,\,N\leq M$ such that  $r(M/N)=M/N$  and  $f\in \,\,\textup{Hom}_{R}(N,Q)$  with $s(M/\textup{ker}(f))=M/\textup{ker}(f)\}$. It is clear that $\mathcal{L}_{(r,s)}=\mathcal{L}_{(\rho_{r},\rho_{s})},$  when $r$ and $s$ are left exact preradicals.

- For any torsion theory $\tau$, we will denote by $\mathcal{L}_{\tau}$ the subclass $\mathcal{L}_{\tau}=\{(M,N,f,Q)\in\M\mid M,N,Q\in R\textup{-Mod},\, N$ is a $\tau$-dense in $M$ and $f\in \textup{Hom}_{R}(N,Q)\}.$  It is clear that $\mathcal{L}_{\tau}=\mathcal{L}_{(\rho_{\tau},\rho_{\infty})},$  when $\tau$ is a hereditary torsion theory.

\end{notations}

\begin{lem}\label{lem:(2.1.9)} Let $\rho_{1}$ and $\rho_{2}$ be two $P$-filters.  Then $\mathcal{L}_{(\rho_{1},\rho_{2})}$ satisfies conditions $\left(\alpha\right)$, $\left(\delta\right)$ and $\left(\mu\right)$.
\end{lem}
\begin{proof} Conditions $\left(\alpha\right)$ and $\left(\delta\right)$  are clear.

Condition $\left(\mu\right)$: Let $(M,N,f,Q)\in\mathcal{L}_{(\rho_{1},\rho_{2})}$    and   let $x\in M,$ thus $(M,N)\in\rho_{1}$ and $f\in \textup{Hom}_{R}(N,Q)$ such that $(M,\textup{ker}(f))\in\rho_{2}$. Since $\rho_{1}$ is a $P$-filter,  $(R,(N:x))\in\rho_{1}.$ It is easy to prove that $\textup{ker}(f_{x})=(\textup{ker}(f):x).$ Since $(M,\textup{ker}(f))\in\rho_{2}$ and $\rho_{2}$   is a $P$-filter, $(R,(\textup{ker}(f):x))\in\rho_{2}$ and hence  $(R,\textup{ker}(f_{x}))\in\rho_{2}$ and this implies that $(R,(N:x),f_{x},Q)\in\mathcal{L}_{(\rho_{1},\rho_{2})}$. Therefore $\mathcal{L}_{(\rho_{1},\rho_{2})}$ satisfies  condition $\left(\mu\right)$.
\end{proof}

 The following corollary is a generalization of Baer's result in \cite{Bae40}, \cite[Proposition 2.1, p. 201]{Ste75},
\cite[Baer's Lemma 2.2, p. 628]{Jir75} and \cite[Theorem 2.4, p. 319]{Bea72}.

\begin{cor}\label{cor:(2.1.12)} Let $\rho_{1}$ and $\rho_{2}$ be any two $P$- filters. Then the following conditions are  equivalent for $R$-module
$M$:

 $\left(1\right)$  $M$ is $\mathcal{L}_{(\rho_{1},\rho_{2})}$-injective;

$\left(2\right)$ for every $(R,I,f,M)\in\mathcal{L}_{(\rho_{1},\rho_{2})}$, there  exists an $R$-homomorphism $g\in \textup{Hom}_{R}(R,M)$ such that $g(a)=f(a),$ for all $a\in I$;

$\left(3\right)$ for each $(R,I,f,M)\in\mathcal{L}_{(\rho_{1},\rho_{2})}$, there exists an element  $m\in M$ such that  $f(r)=rm,\, \forall r\in I$.
\end{cor}
\begin{proof} By Lemma~\ref{lem:(2.1.9)} and Theorem~\ref{thm:(2.1.11)}.
\end{proof}

 The following characterization of  $\mathcal{L}$-injectivity  is a generalization of \cite[Proposition 1.4, p. 3]{Smi97} and
 \cite[Proposition 2.1.3, p. 53]{Cri04}.

\begin{prop}\label{prop:(2.1.13)} Consider the following three conditions for $R$-module $M$:

$\left(1\right)$  $Q$ is $\mathcal{L}$-injective;

$\left(2\right)$ For every $(M,N,f,Q)\in\mathcal{L}$ with $N\leq_{e}M$,    the homomorphism $f$ extends to a homomorphism from $M$ to $Q$;

 $\left(3\right)$  For every $(R,I,f,Q)\in\mathcal{L}$ with $I\leq_{e}R$,  the homomorphism $f$ extends to a homomorphism from $R$ to $Q$.

 Then $\left(1\right)$ implies $\left(2\right)$, $\left(2\right)$
implies $\left(3\right)$ and, if $\mathcal{L}$ satisfies  conditions $\left(\alpha\right)$
and $\left(\mu\right)$, then $\left(3\right)$ implies$\left(1\right)$.
\end{prop}
\begin{proof}$\left(1\right)\Rightarrow\left(2\right)$ and $\left(2\right)\Leftrightarrow\left(3\right)$ are obvious.

$\left(3\right)\Rightarrow\left(1\right)$ Let $\mathcal{L}$ satisfy  $\left(\alpha\right)$ and $\left(\mu\right)$
and  let $(R,I,f,Q)\in\mathcal{L}$. Let $I^{c}$ be a complement left ideal of $I$  in $R$
and let $C=I\oplus I^{c}$. Thus by \cite[Proposition 5.21, p. 75]{AnFu92},
 $C\leq_{e}R$. Define $g:C=I\oplus I^{c}\rightarrow Q$
 by $g(a+b)=f(a)\;,\;$ $\forall a\in I$ and $\forall b\in I^{c}$.
It is clear that $g$ is a well-defined homomorphism and $(R,I,f,Q)\preceq(R,C,g,Q)$. Since $\mathcal{L}$ satisfies condition  $\left(\alpha\right)$,
 $(R,C,g,Q)\in\mathcal{L}$. By hypothesis, there exists a homomorphism
$h:R\rightarrow Q$ such that $(h \upharpoonright C)=g.$ Thus $(h \upharpoonright I)=(g \upharpoonright I)=f$
and this implies that $Q$ is $\mathcal{L}$-injective, by Theorem~\ref{thm:(2.1.11)}.
\end{proof}

In the following theorem we extend a  characterization due to \cite[Theorem 2, p. 8]{Yi98} of $\mathcal{L}$-injective modules over commutative Noetherian rings.

\begin{thm}\label{thm:(2.1.14)} Let $R$ be a commutative Noetherian ring, let  $M$ be an $R$-module and suppose that $\mathcal{L}$ satisfies conditions $\left(\alpha\right)$ and $\left(\mu\right)$. Then  $M$ is $\mathcal{L}$-injective if and only if for every $(R,I,f,M)\in\mathcal{L}$ with $I$ is a prime ideal of $R$,   the homomorphism f extends to a homomorphism
from $R$ to $M$.
\end{thm}

\begin{proof}
$(\Longrightarrow)$ This is obvious.

$(\Longleftarrow)$
Let  $(B,A,f,M)\in\mathcal{L}$ and let  $S=\{(C,\varphi)\mid A\leq C\leq B\:,\:\varphi\in \textup{Hom}_{R}(C,M)$
such that $(\varphi \upharpoonright A)=f\:\}.$ Define on $S$ a partial order $\preceq$
by
\[
(C_{1},\varphi_{1})\preceq(C_{2},\varphi_{2})\Longleftrightarrow C_{1}\leq C_{2}\; and\;(\varphi_{2}\upharpoonright C_{1})=\varphi_{1}.\]

As in the proof of Theorem~\ref{thm:(2.1.11)}, we can prove that $S$ has a maximal element, say $(B^{\prime},g^{\prime})$. We will prove that $B^{\prime}=B$. Suppose that there exists $x\in B\setminus B^{\prime}$. By \cite[Theorem 1, p. 8]{Yi98}, there  exists an element $r_{0}\in R$ such that $(B^{\prime}:r_{0}x)$ is a prime ideal in $R$ and $r_{0}x\notin B^{\prime}$. It is clear that $(B,A,f,M)\preceq(B,B^{\prime},g^{\prime},M)$. Since
$(B,A,f,M)\in\mathcal{L}$ and $\mathcal{L}$ satisfies condition $\left(\alpha\right)$, it follows that $(B,B^{\prime},g^{\prime},M)\in\mathcal{L}.$ Since
$\mathcal{L}$ satisfies condition  $\left(\mu\right)$,  $(R,(B^{\prime}:b),g_{b}^{\prime},M)\in\mathcal{L},\, \forall b\in B.$ Put $y=r_{0}x$, thus $y\in B\setminus B^{\prime}$
and hence $(R,(B^{\prime}:y),g_{y}^{\prime},M)\in\mathcal{L}$. By hypothesis, there exists a homomorphism  $g:R\rightarrow M$ such that
$g(r)=g_{y}^{\prime}(r)=g^{\prime}(ry),$ $\forall r\in(B^{\prime}:y)$. Define $\psi:B^{\prime}+Ry\rightarrow M$
by $\psi(b+ry)=g^{\prime}(b)+g(r),\,\forall b\in B^{\prime},\,\forall r\in R.$
As in the proof of Theorem~\ref{thm:(2.1.11)}, we can prove that $\psi$ is a
well-defined homomorphism and  $(B^{\prime},g^{\prime})\preceq(B^{\prime}+Ry,\psi)$. Since $(B^{\prime}+Ry,\psi)\in S$ and $B^{\prime}\subsetneqq B^{\prime}+Ry$, we have a contradiction to maximality of $(B^{\prime},g^{\prime})$
in $S$. Hence $B^{\prime}=B$ and this mean that there exists a
homomorphism $g^{\prime}:B\rightarrow M$ such that $(g^{\prime}$$\upharpoonright$$A)=f$.
Thus M is $\mathcal{L}$-injective.
\end{proof}

\begin{cor}\label{cor:(2.1.15)} Let $\rho_{1}$ and $\rho_{2}$ be any two $P$-$filters$,let $R$ be a commutative Noetherian ring and let  $M$ be an $R$-module. Then $M$ is $\mathcal{L}_{(\rho_{1},\rho_{2})}$-injective if and only if  for every  $(R,I,f,M)\in\mathcal{L}_{(\rho_{1},\rho_{2})}$ with $I$ is a prime ideal of $R$,  the homomorphism $f$ extends to a homomorphism from $R$ to $M$.
\end{cor}
\begin{proof} By Lemma~\ref{lem:(2.1.9)} and Theorem~\ref{thm:(2.1.14)}.
\end{proof}

\begin{cor}\label{cor:(2.1.16)} \textbf{(\cite[Theorem 2, p. 8]{Yi98})}  Let $R$ be a commutative Noetherian ring, let  $M$ be an $R$-module. Then  $M$ is injective if and only if every homomorphism $f:I\rightarrow M$ with  $I$ is a prime ideal
of $R$, can be extended to a homomorphism from $R$ to $M$.
\end{cor}
\begin{proof} By taking the two $P$-filters $\rho_{1} =\rho_{2}$$=\Re$   and applying Corollary~\ref{cor:(2.1.15)}.
\end{proof}

\section{$\mathcal{L}$-$M$-Injectivity and $s$-$\mathcal{L}$-$M$-Injectivity}

$\quad$ In this section, we introduce the concepts of $\mathcal{L}$-$M$-injective modules and $s$-$\mathcal{L}$-$M$-injective modules as generalizations of $M$-injective modules and give some results about them.

\begin{defn}\label{defn:(2.2.1)} Let $M,Q\in R\textup{-Mod}$. A module $Q$ is said to be $\mathcal{L}$-$M$-injective, if
for every  $(M,N,f,Q)\in\mathcal{L}$,  the homomorphism $f$ extends to a homomorphism from $M$ to $Q$. \, A module \, $Q$ \, is said to be $\mathcal{L}$-quasi-injective, if $Q$ is $\mathcal{L}$-$Q$-injective.
 \end{defn}

$\vphantom{}$

Let $M,Q\in R\textup{-Mod}$, it is well-known that a module $Q$ is $M$-injective if and only if $f(M)\leq Q$, for every homomorphism $f:E(M)\rightarrow E(Q)$ \cite[Lemma 1.13, p. 7]{MoMu90}.

$\vphantom{}$

For an analogous result for $\mathcal{L}$-$M$-injectivity we first fix the following condition.

$\vphantom{}$

$\left(E_{\mathcal{L}}\right)$:  Let $\mathcal{L}$
be a subclass of $\M$. Then a module $M$ satisfies condition $\left(E_{\mathcal{L}}\right)$ if  $M$ has an $\mathcal{L}$-injective envelope which is unique up to $M$-isomorphism and   $(E_{\mathcal{L}}(M),N,f,Q) \in\mathcal{L}$
whenever $(M,N,f,Q)\in\mathcal{L}$.

$\vphantom{}$

The next theorem is the first main result of this section in which we give a generalization of \cite[Lemma 1.13, p. 7]{MoMu90}
and \cite[Theorem 2.1, p. 34]{Cri01}.

\begin{thm}\label{thm:(2.2.5)} Let $M,Q\in R\textup{-Mod}$ and let $\mathcal{L}$ satisfy conditions $\left(\alpha\right)$, $\left(\beta\right)$ and $\left(\gamma\right)$. Consider the following two conditions

$\left(1\right)$ $Q$ is $\mathcal{L}$-$M$-injective.

$\left(2\right)$  $f(M)\leq Q$,  for all $f\in\textup{Hom}_{R}(E_{\mathcal{L}}(M),E_{\mathcal{L}}(Q))$
with $(M,L,f$$\upharpoonright$$L,Q)\in\mathcal{L}$ where $L=\{m\in M\,\mid\:f(m)\in Q\}=M\bigcap f^{-1}(Q)$.

Then $\left(1\right)$ implies $\left(2\right)$ and, if $M$ satisfies condition  $\left(E_{\mathcal{L}}\right)$, then $\left(2\right)$
implies $\left(1\right)$.
\end{thm}
\begin{proof}
$\left(1\right)\Rightarrow\left(2\right)$ Let $f\in \textup{Hom}_{R}(E_{\mathcal{L}}(M),E_{\mathcal{L}}(Q))$
with $(M,L,f$$\upharpoonright$$L,Q)\in\mathcal{L}$, where $L=\{m\in M\,\mid\:f(m)\in Q\}=M\bigcap f^{-1}(Q)$.
Define $g:L\rightarrow Q$ by $g(a)=f(a)$, \,$\forall a\in L.\; \,\,(\: i.e.,\,\, g=f$$\upharpoonright$$L)$. It is clear that $g$ is a homomorphism
and $(M,L,g,Q)\in\mathcal{L}$. By $\mathcal{L}$-$M$-injectivity of $Q$, there exists a homomorphism
$h:M\rightarrow Q$ such that $(h$$\upharpoonright$$L)=g.$ Since $Q\bigcap(f$-$h)(M)=0$ and $Q$ is an
essential submodule of $E_{\mathcal{L}}(Q)$ (by \cite[Theorem 1.19 ,p. 627]{Jir75}), it follows that $(f$-$h)(M)=0$
and this implies that $f(M)=h(M)\leq Q$.

$\left(2\right)\Rightarrow\left(1\right)$ Let $M$ satisfy condition $\left(E_{\mathcal{L}}\right)$  and let $(M,N,f,Q)\in\mathcal{L}$, thus $(E_{\mathcal{L}}(M),N,f,Q)\in\mathcal{L}$. Since $\mathcal{L}$ satisfies condition $\left(\beta\right)$, we have that  $(E_{\mathcal{L}}(M),N,if,E_{\mathcal{L}}(Q))\in\mathcal{L}$, where $i$ is the inclusion mapping from $Q$ into  $E_{\mathcal{L}}(Q)$.
By $\mathcal{L}$-injectivity of $E_{\mathcal{L}}(Q)$, there exists
a homomorphism $h:E_{\mathcal{L}}(M)\rightarrow E_{\mathcal{L}}(Q)$
such that $h(n)=f(n)$, \, $\forall n\in N.$ \,Let $L=\{m\in M\:\mid\:h(m)\in Q\}$. We
will prove that $(M,L,g,Q)\in\mathcal{L}$, where $g=h \upharpoonright L$. Let $x\in N$, thus $h(x)=f(x)\in Q$ and hence $x\in L$. Thus $N\leq L$
and $(g \upharpoonright N)=f$. Thus $(M,N,f,Q)\preceq(M,L,g,Q)$. Since $\mathcal{L}$ satisfies  condition $\left(\alpha \right)$,  $(M,L,g,Q)\in\mathcal{L}$.  By hypothesis, we have that $h(M)\leq Q$ and hence  $h^{\prime}=h$$\upharpoonright  M:M\rightarrow Q$ is such that $(h^{\prime} \upharpoonleft N)=f$. Thus   $Q$ is an $\mathcal{L}$-$M$-injective module.
\end{proof}

\begin{cor}\label{cor:(2.2.6)} Let $M,Q\in R\textup{-Mod}$ and let $\rho_{1}$ and $\rho_{2}$
be any two $P$-filters. If $M$ satisfies condition  $(E_{\mathcal{L}_{(\rho_{1},\rho_{2})}})$,
  then the following two conditions are equivalent.

$\left(1\right)$  $Q$ is $\mathcal{L}_{(\rho_{1},\rho_{2})}$-$M$-injective;

$\left(2\right)$  $f(M)\leq Q$, for all $f\in \textup{Hom}_{R}(E_{\mathcal{L}_{(\rho_{1},\rho_{2})}}(M),E_{\mathcal{L}_{(\rho_{1},\rho_{2})}}(Q))$
with $(M,L,f$$\upharpoonright$$L,Q)\in\mathcal{L}$ where $L=\{m\in M\,\mid\:f(m)\in Q\}=M\bigcap f^{-1}(Q)$.
\end{cor}
\begin{proof}
By Lemma~\ref{lem:(2.1.9)} and Theorem~\ref{thm:(2.2.5)}.
\end{proof}

 Let $M,Q\in R\textup{-Mod}$ and let $\tau$ be any hereditary torsion theory. A module $Q$ is $s$-$\tau$-$M$-injective if, for any $N\leq M$, any homomorphism from a $\tau$-dense  submodule of $N$ to $Q$ extends to a homomorphism from $N$ to $Q$  \cite[Definition 14.6, p. 65]{Cha06}.

 As a generalization  of $s$-$\tau$-$M$-injectivity and hence of  $M$-injectivity
we introduce the concept of $s$-$\mathcal{L}$-$M$-injectivity as follows.

\begin{defn}\label{defn:(2.2.11)} Let $M,Q\in R\textup{-Mod}$. A module $Q$ is said to be $s$-$\mathcal{L}$-$M$-injective
if $Q$ is $\mathcal{L}$-$N$-injective, for all $N\leq M$. A module $Q$ is said to be $s$-$\mathcal{L}$-quasi-injective if $Q$ is
$s$-$\mathcal{L}$-$Q$-injective.
\end{defn}

 Fuchs in \cite{Fuc69} has obtained a condition similar to Baer's Criterion that characterizes quasi-injective modules, Bland
in \cite{Bla90} has generalized that to $s$-$\tau$-quasi-injective modules, and Charalambides in \cite{Cha06} has generalized that
to $s$-$\tau$-$M$-injective modules.

$\vphantom{}$

Our next aim is to generalize Fuchs's condition once again in order to characterize $s$-$\mathcal{L}$-$M$-injective modules. We begin
with the following condition.

$\vphantom{}$

 $\left(\mathcal{L}\right)$:  Let $\mathcal{L}$ be a subclass of $\M$ and let $M$ be a module . Then $M$ satisfies condition $\left(\mathcal{L}\right)$ if for every $(B,A,f,Q)\in\mathcal{L}$, then $(Rm,(A:x)m,f_{(x,m)},Q)\in\mathcal{L}$, for all $m\in M$ and $x\in B$ with $ann_{R}(m)\subseteq(\textup{ker}(f):x)$, where $f_{(x,m)}:(A:x)m\rightarrow Q$ is a well-defined homomorphism
defined by $f_{(x,m)}(rm)=f(rx)$, for all $r\in(A:x)$ .

 A subclass $\mathcal{L}$ of $\M$ is said to be fully
subclass if every $R$-module satisfies condition $\left(\mathcal{L}\right)$.

$\vphantom{}$

\begin{example}\label{example:(2.2.13)} All of the following subclasses of $\M$ are fully subclasses.

$\left(1\right)$ $\mathcal{L}_{(T,F)}$ where $T$ and $F$ are nonempty
classes of modules closed under submodules and  homomorphic
images.

$\left(2\right)$  $\mathcal{L}=\M$.

 $\left(3\right)$ $ \mathcal{L}_{\tau}$,  where $\tau$  is a hereditary torsion theory.

$\left(4\right)$ $\mathcal{L}_{(\rho,\sigma)}$, where $\rho$ and $\sigma$ are left exact preradicals.
\end{example}

\begin{proof} $\left(1\right)$ Let $(B,A,f,Q)\in\mathcal{L}_{(T,F)}$ and let $m\in M, x\in B$
such that $ann_{R}(m)\subseteq(\textup{ker}(f):x)$. By Lemma~\ref{lem:(2.1.9)}  we have that $\quad$$\mathcal{L}_{(T,F)}$  satisfies condition
$\left(\mu\right)$  and  this  implies  that $(B,(A:x),f_{x},Q)\in\mathcal{L}_{(T,F)}$.
It is clear that $(R/(Im:m))\simeq(Rm/Im)$ and $(R/(Jm:m))\simeq(Rm/Jm)$
where $I=(A:x)$ and $J=\textup{ker}(f_{x})$ . Since $\mathcal{L}_{(T,F)}=\mathcal{L}_{(\rho_{1},\rho_{2})}$, where $\rho_{1}$ and $\rho_{2}$ defined by $\rho_{1}=\{(M,N)\in\Re\mid N\leq M$ such that $M/N\in T,\, M\in R\textup{-Mod}\}$ and $\rho_{2}=\{(M,N)\in\Re\mid N\leq M$
such that $M/N\in F,\, M\in R\textup{-Mod}\}$, it follows that $(R,I)\in\rho_{1}$ and $(R,J)\in\rho_{2}$. Since $I\leq(Im:m)\leq R$ and $J\leq(Jm:m)\leq R$
and $\rho_{1}$, $\rho_{2}$ are $P$-filters (by example~\ref{example:(2.1.8)}),  $(R,(Im:m))\in\rho_{1}$ and $(R,(Jm:m))\in\rho_{2}$
and this implies that $(R/(Im:m))\in T$ and $(R/(Jm:m))\in F$. Since $T$ and $F$ are closed under homomorphic images,  $(Rm/Im)\in T$
and $(Rm/Jm)\in F$. Since $\textup{ker}(f_{x})=(\textup{ker}(f):x)$  and  $\textup{ker}(f_{(x,m)})=(\textup{ker}(f):x)m$, we have that
 $\textup{ker}(f_{(x,m)})=\textup{ker}(f_{x})m=Jm$ and this implies that $(Rm/(A:x)m)\in T$ and $(Rm/\textup{ker}(f_{(x,m)}))\in F$.Thus $(Rm,(A:x)m,f_{(x,m)},Q)\in\mathcal{L}_{(T,F)}$
and hence $\mathcal{L}_{(T,F)}$ is a fully subclass.

$\left(2\right),\left(3\right)$ and $\left(4\right)$ are special cases of $\left(1\right)$.
\end{proof}

$\vphantom{}$

 In following proposition,  we generalize \cite[Proposition 14.12, p. 66]{Cha06},
 \cite[Proposition 1, p. 1954]{Bla90} and Fuchs's result in \cite{Fuc69}, and it is necessary for our version of the Generalized Fuchs  criterion
.

\begin{prop}\label{prop:(2.2.14)} Consider the following statements, where $M,Q\in R\textup{-Mod}$:

 $\left(1\right)$ $Q$ is $s$-$\mathcal{L}$-$M$-injective;

$\left(2\right)$  if $m\in M$ with $(Rm,K,f,Q)\in\mathcal{L}$, then  the homomorphism $f$ extends to a homomorphism from $Rm$ to $Q$;

$\left(3\right)$  if $K\leq N$ are modules, not necessarily submodules of $M$ such that  $(N,K,f,Q)\in\mathcal{L}$ and $\Omega(N)\subseteq\Omega(M)$,  then the homomorphism $f$ extends to a homomorphism from $N$ to $Q$.

 Then $\left(1\right)$ implies $\left(2\right)$ and $\left(3\right)$
implies $\left(1\right)$. Moreover, if $\mathcal{L}$ satisfies condition
$\left(\alpha\right)$  and $M$ satisfies condition $\left(\mathcal{L}\right)$, then all above statements are equivalent.
\end{prop}

\begin{proof}
$\left(1\right)\Rightarrow\left(2\right)$ Let $m\in M$  with $(Rm,K,f,Q)\in\mathcal{L}$. Thus $Q$ is
$\mathcal{L}$-$Rm$-injective, since $Q$ is $s$-$\mathcal{L}$-$M$-injective and hence  there exists a homomorphism $g:Rm\rightarrow Q$ such that $(g$$\upharpoonright$$K)=f$.

 $\left(2\right)\Rightarrow\left(3\right)$ Let $\mathcal{L}$ satisfy  condition $\left(\alpha\right)$  and
$M$ satisfy condition $\left(\mathcal{L}\right)$. Let $K\leq N$ be modules, not necessarily submodules of $M$ with \, $(N,K,f,Q)\in\mathcal{L}$ and $\Omega(N)\subseteq\Omega(M)$. \,Let \, $S=\{(C,\varphi)\mid K\leq C\leq N, \, \varphi\in \textup{Hom}_{R}(C,M)$
such that $(\varphi \upharpoonright K)=f\}.$ Define on $S$ a partial order $\preceq$ by
\[
(C_{1},\varphi_{1})\preceq(C_{2},\varphi_{2})\Longleftrightarrow C_{1}\leq C_{2}\; and\;(\varphi_{2}\upharpoonright C_{1})=\varphi_{1}\]
As in the proof of Theorem~\ref{thm:(2.1.11)}, we can prove that $S$ has a maximal element, say $(X,h)$. It suffices to show that $X=N$. Suppose that
there exists $n\in N\setminus X$. It is clear that $(N,K,f,Q)\preceq(N,X,h,Q)$.
Since $(N,K,f,Q)\in\mathcal{L}$ and $\mathcal{L}$ satisfies condition $\left(\alpha\right)$,  $(N,X,h,Q)\in\mathcal{L}.$
Since $ann_{R}(n)\in\Omega(N)$ and $\Omega(N)\subseteq\Omega(M)$ (by assumption),  $ann_{R}(n)\in\Omega(M)$
and this implies that there exists $m\in M$ such that $ann_{R}(m)\subseteq ann_{R}(n)$. Since $ann_{R}(n)\subseteq(\textup{ker}(h):n)$, $ann_{R}(m)\subseteq(\textup{ker}(h):n)$. Since $m\in M$ and $n\in N\setminus X$ such that $ann_{R}(m)\subseteq(\textup{ker}(h):n)$
and since $M$ satisfies condition $\left(\mathcal{L}\right)$,  $(Rm,(X:n)m,h_{(n,m)},Q)\in\mathcal{L}$.
By hypothesis, there exists a homomorphism $\varphi^{*}:Rm\rightarrow Q$
such that $\varphi^{*}(am)=h_{(n,m)}(am)$, for all $am\in(X:n)m$. Define $h^{*}:X+Rn\rightarrow Q$ by $h^{*}(x+rn)=h(x)+\varphi^{*}(rm)$, \,$\forall x\in X$ and $\forall r\in R$. Clearly $h^{*}$ is a well-defined homomorphism. For all $a\in K$, we have that $h^{*}(a)=h^{*}(a+0.n)=h(a)+\varphi^{*}(0.m)=h(a)=f(a)$ and hence $(h^{*}$$\upharpoonright$$K)=f$. Since $K\leq X+Rn\leq N$,  $(X+Rn,h^{*})\in S$. Since $(h^{*} \upharpoonright X)=h$ and $X\leq X+Rn\leq N$,  $(X,h)\preceq(X+Rn,h^{*})$.
Since $n\in X+Rn$ and $n\notin X$, $X\subsetneqq X+Rn$ and this contradicts the maximality of $(X,h)$ in $S$. Thus $X=N$ and this implies that there exists a homomorphism $h:N\rightarrow Q$ such that $(h$$\upharpoonright$$K)=f$.

 $\left(3\right)\Rightarrow\left(1\right)$ Let $N\leq M$  with $(N,K,f,Q)\in\mathcal{L}$. Let  $I\in\Omega(N),$ thus  there  exists an element $n\in N$  such  that $ann_{R}(n)\subseteq I$ and hence there exists an element $n\in M$ such that $ann_{R}(n)\subseteq I$ and this implies that $I\in\Omega(M)$ and so $\Omega(N)\subseteq\Omega(M)$. By hypothesis, there exists a homomorphism $g:N\rightarrow Q$ such
that $(g$$\upharpoonright$$K)=f$. Thus $Q$ is $\mathcal{L}$-$N$-injective module, for all $N\leq M$ and this implies that $Q$ is $s$-$\mathcal{L}$-$M$-injective.
\end{proof}

 Follow we give the last main result of this section in which we generalize \cite[Proposition 14.13, p. 68]{Cha06},
 \cite[Proposition 2, p. 1955]{Bla90} and \cite[Lemma 2, p. 542]{Fuc69}. It is our version of Generalized Fuchs criterion.

\begin{prop}\label{prop:(2.2.15)} \textbf{{(Generalized Fuchs criterion) }} Consider the following conditions, where $M,Q\in R\textup{-Mod}$.

$\left(1\right)$ $Q$ is $s$-$\mathcal{L}$-$M$-injective;

$\left(2\right)$  for each $(R,I,f,Q)\in\mathcal{L}$ with
$\textup{ker}(f)\in\Omega(M)$,  the homomorphism $f$ extends to a homomorphism
from $R$ to $Q$;

$\left(3\right)$  for each $(R,I,f,Q)\in\mathcal{L}$ with
$\textup{ker}(f)\in\Omega(M)$, \, there exists an element  \, $x\in Q$  such that  $f(a)=ax,  \;\forall a\in I$.

 Then $\left(2\right)\Leftrightarrow\left(3\right)$ and if $M$ satisfies condition $\left(\mathcal{L}\right)$ then $\left(1\right)$
implies $\left(2\right)$. Moreover, if $\mathcal{L}$ satisfies conditions
$\left(\alpha\right)$ and $\left(\mu\right)$, then $\left(2\right)$ implies $\left(1\right)$.
\end{prop}
\begin{proof}$\left(2\right)\Leftrightarrow\left(3\right)$ This is obvious.

$\left(1\right)\Rightarrow\left(2\right)$ Let $M$ satisfy condition $\left(\mathcal{L}\right)$
 and let $(R,I,f,Q)\in\mathcal{L}$ with $\textup{ker}(f)\in\Omega(M)$. Thus there exists an element $m\in M$ such that $ann_{R}(m)\subseteq \textup{ker}(f).$ Since $\textup{ker}(f)=(\textup{ker}(f):1)$ where 1 is the identity element of $R$,  $ann_{R}(m)\subseteq(\textup{ker}(f):1)$.
Since  $M$ satisfies condition $\left(\mathcal{L}\right)$,  $(Rm,(I:1)m,f_{(1,m)},Q)\in\mathcal{L}$ and hence   $(Rm,Im,f_{(1,m)},Q)\in\mathcal{L}$.  Since $Q$ is $s$-$\mathcal{L}$-$M$-injective, it follows Proposition~\ref{prop:(2.2.14)} implies that there exists a homomorphism $h:Rm\rightarrow Q$ such that $h\circ i_{2}=f_{(1,m)}$, where $i_{2}$ is the inclusion mapping from $Im$ into $Rm$. Define $v_{1}:I\rightarrow Im$ by $v_{1}(a)=am$, \, $\forall a\in I$
and define $v_{2}:R\rightarrow Rm$ by $v_{2}(r)=rm$, $\forall r\in R$. It is clear that $v_{1}$ and $v_{2}$ are homomorphisms and for
all $a\in I$, we have that $(v_{2}\circ i_{1})(a)=(i_{2}\circ v_{1})(a)$, where $i_{1}$ is the inclusion mapping from $I$ into $R$.
 Define $g:R\rightarrow Q$ by $g(r)=(h\circ v_{2})(r)$,\, $\forall r\in R$. It is clear that $g$ is a homomorphism and for all $a\in I$ we
have that $(g\circ i_{1})(a)=f_{(1,m)}(v_{1}(a))=f_{(1,m)}(am)=f(a.1)=f(a)$. Thus there exists a homomorphism $g:R\rightarrow Q$ such that $(g$$\upharpoonright$$I)=f$.

 $\left(2\right)\Rightarrow\left(1\right)$ Let $\mathcal{L}$ satisfy conditions $\left(\alpha\right)$ and $\left(\mu\right)$.
Let $K\leq N\leq M$ such that $(N,K,f,Q)\in\mathcal{L}$ and let  $S=\{(C,\varphi)\mid K\leq C\leq N,\:\varphi\in \textup{Hom}_{R}(C,M)$ such that $(\varphi$$\upharpoonright$$K)=f\}.$ Define on $S$ a partial order $\preceq$ by
\[
(C_{1},\varphi_{1})\preceq(C_{2},\varphi_{2})\Longleftrightarrow C_{1}\leq C_{2}\; and\;(\varphi_{2}\upharpoonright C_{1})=\varphi_{1}\]
As in the proof of Theorem~\ref{thm:(2.1.11)}, we can prove that $S$ has a maximal element, say $(X,h)$. It suffices to show that $X=N$. Suppose that there exists $n\in N\setminus X$. It is clear that $(N,K,f,Q)\preceq(N,X,h,Q)$. Since $(N,K,f,Q)\in\mathcal{L}$ and $\mathcal{L}$ satisfies condition$\left(\alpha\right)$,   $(N,X,h,Q)\in\mathcal{L}.$ Since $\mathcal{L}$ satisfies condition $\left(\mu\right)$ and $n\in N\setminus X$, $(R,(X:n),h_{n},Q)\in\mathcal{L}$. Since $(0:n)\subseteq \textup{ker}(h_{n})$ and $n\in M$,  $\textup{ker}(h_{n})\in\Omega(M)$. By hypothesis, there exists a homomorphism
$\varphi^{*}:R\rightarrow Q$ such that $(\varphi^{*}$$\upharpoonright$$(X:n))=h_{n}$. Define $h^{*}:X+Rn\rightarrow Q$ by $h^{*}(x+rn)=h(x)+\varphi^{*}(r),\,\forall x\in X, \,\forall r\in R$. We can prove that $h^{*}$ is a well-defined homomorphism, $(X,h)\preceq(X+Rn,h^{*})$ and $(X+Rn,h^{*})\in S$. Since $n\in X+Rn$ and $n\notin X$,  $X\subsetneqq X+Rn$ and this
contradicts the  maximality of $(X,h)$ in $S$. Thus $X=N$ and this implies that there exists a homomorphism $h:N\rightarrow Q$ such that $(h$$\upharpoonright$$K)=f$. Thus $Q$ is $\mathcal{L}$-$N$-injective module, \,for all $N\leq M$ and hence $Q$ is $s$-$\mathcal{L}$-$M$-injective
$R$-module.

\end{proof}

\section{Direct Sums of $\mathcal{L}$-Injective Modules}

The direct sums of $\mathcal{L}$-injective modules is  not $\mathcal{L}$-injective in general, for example: let  $\{T_{i}\}_{i\in I}$  be a family of rings with unit and let $R=\prod_{i\in I}T_{i}$ be the ring product of the family $ \{T_{i}\}_{i\in I}$, where addition and multiplication are define componentwise. Let $A=\sqcup_{i\in I}T_{i}$ \, the  direct sum of $T_{i},\,\forall i\in I$.  If each $_{T_{i}}T_{i}$ is injective, $\forall \, i\in I$ and $I$ is infinite, then
$_{R}A$ is a direct sum of injective modules, but $_{R}A$ is not itself injective, by  \cite[p. 140]{Kas82}. Hence  we have that $_{R}A$ is a direct sum of $\mathcal{L}$-injective modules, but $_{R}A$ is not itself $\mathcal{L}$-injective where $\mathcal{L}=\mathcal{M}$.

$\vphantom{}$

In following we study conditions under which the class of $\mathcal{L}$-injective modules is closed under direct sums.

$\vphantom{}$

Let $\{E_{\alpha}\}_{\alpha\in A}$ be a family of modules and let $E=\bigoplus_{\alpha\in A}E_{\alpha}.$ For any $x=(x_{\alpha})_{\alpha\in A}\in E,$ we define the the support of $x$ to be the set $\{\alpha\in A$$\mid x_{\alpha}\neq0\}$ and denote it by $\textup{supp}(x).$ For any $X\subseteq E,$ we define  $\textup{supp}(X)$ to be the set $\underset{x\in X}{\bigcup}\textup{supp}(x)=\{\alpha\in A\mid(\exists \,x\in X)\,x_{\alpha}\neq0\}.$

$\vphantom{}$

The following condition will be useful later.

$\vphantom{}$

$\left(F\right)$: Let $\{E_{\alpha}\}_{\alpha\in A}$ be a family of modules, where $A$ is an infinite index set and let $\mathcal{L}$
be a subclass of $\M.$ We say that $\mathcal{L}$ satisfies condition $\left(F\right)$ for a family $\{E_{\alpha}\}_{\alpha\in A},$ if for any
$(R,I,f,\bigoplus_{\alpha\in A}E_{\alpha})\in\mathcal{L},$ then $\textup{supp}(\textup{im}(f))$ is finite.

\begin{lem}\label{lem:(2.3.1)} Let $A$ be any index set and let $C$ be any countable subset of $A,$ and let $\{E_{\alpha}\}_{\alpha\in A}$ be any family of modules. Define $\pi_{C}:\bigoplus_{\alpha\in A}E_{\alpha}\rightarrow\bigoplus_{\alpha\in C}E_{\alpha}$ by  $\pi_{C}(x)=x_{C}$,   \,for all $x\in\bigoplus_{\alpha\in A}E_{\alpha}$ where $\pi_{\alpha}(x_{C})=\pi_{\alpha}(\pi_{C}(x))=\begin{cases}
\begin{array}{cc} \pi_{\alpha}(x) &  if\;\alpha\in C\\ 0 &  if\;\alpha\notin C\end{array},\end{cases} \,\forall\alpha\in A,$
where  $\pi_{\alpha}$ is the $\alpha th$ projection  homomorphism. Then $\pi_{C}$ is a well-defined homomorphism and if $x\in\bigoplus_{\alpha\in C}E_{\alpha},$ then $\pi_{C}(x)=x.$
\end{lem}
\begin{proof}An easy check.
\end{proof}

\begin{lem}\label{lem:(2.1.3)} Let $\{M_{i}\}_{i\in I}$ be any family of modules. If  $M_{i}$ is $\mathcal{L}$-injective,   $\forall \,i\in I$ and $\mathcal{L}$ satisfies condition $\left(\lambda\right)$,  then $\prod_{i\in I}M_{i}$ is $\mathcal{L}$-injective.
\end{lem}
\begin{proof} This is obvious.
\end{proof}

 The following corollary is immediately from Lemma~\ref{lem:(2.1.3)}.

\begin{cor}\label{cor:(2.1.5)} Let $\mathcal{L}$  satisfy condition $\left(\lambda\right)$ and let  $\{M_{i}\}_{i\in I}$ be any family of $\,$$\mathcal{L}$-injective modules.  If $I$ is a finite set, then $\bigoplus_{i\in I}M_{i}$
is $\mathcal{L}$-injective.
 \end{cor}

\begin{lem}\label{lem:(2.3.2)} Let $\mathcal{L}$ satisfy the conditions $\left(\alpha\right),\left(\mu\right)$ and $\left(\delta\right)$  and let $\{E_{\alpha}\}_{\alpha\in A}$ be any family of $\mathcal{L}$-injective modules, where $A$ is an infinite index set. If $\mathcal{L}$ satisfies condition $\left(F\right)$ for a family $\{E_{\alpha}\}_{\alpha\in A}$, then $\bigoplus_{\alpha\in A}E_{\alpha}$ is an $\mathcal{L}$-injective module.
\end{lem}
\begin{proof}   Suppose that $\mathcal{L}$ satisfies condition $\left(F\right)$  for
the family $\{E_{\alpha}\}_{\alpha\in A}$ and let $(R,I,f,\bigoplus_{\alpha\in A}E_{\alpha})\in\mathcal{L}.$ Thus $\textup{supp}(\textup{im}(f))$ is finite and this implies that $f(I)\subseteq\bigoplus_{\alpha\in F}E_{\alpha},$ where $F$ is a finite subset of $A.$ Since $E_{\alpha}$ is $\mathcal{L}$-injective, \,$\forall\alpha\in F,$ it follows Corollary~\ref{cor:(2.1.5)} implies that $\bigoplus_{\alpha\in F}E_{\alpha}$ is $\mathcal{L}$-injective.  Define $\pi_{F}:\bigoplus_{\alpha\in A}E_{\alpha}\rightarrow\bigoplus_{\alpha\in F}E_{\alpha}$ by \, $\pi_{F}(x)=x_{F},$ \, for all $x\in\bigoplus_{\alpha\in A}E_{\alpha}$, where $\pi_{\alpha}(x_{F})=\pi_{\alpha}(\pi_{F}(x))=\begin{cases}
\begin{array}{cc}
\pi_{\alpha}(x) &  if\;\alpha\in F\\
0 &  if\;\alpha\notin F\end{array},\end{cases}\,\forall\alpha\in A,$ \, where $\pi_{\alpha}$  is
the $\alpha$th projection homomorphism.  By Lemma~\ref{lem:(2.3.1)}, we have that $\pi_{F}$ is a well-defined homomorphism. Since $(R,I,f,\bigoplus_{\alpha\in A}E_{\alpha})\in\mathcal{L}$ and $\mathcal{L}$ satisfies condition $\left(\delta\right)$,
$(R,I,\pi_{F}\circ f,\bigoplus_{\alpha\in F}E_{\alpha})\in\mathcal{L}.$
 By $\mathcal{L}$-injectivity of $\bigoplus_{\alpha\in F}E_{\alpha},$
there exists a homomorphism $g:R\rightarrow\bigoplus_{\alpha\in F}E_{\alpha}$ such that $g(a)=(\pi_{F}\circ f)(a),$ $\forall a\in I.$ Put $g^{\prime}=i_{1}\circ g:R\rightarrow$$\bigoplus_{\alpha\in A}E_{\alpha},$ where $i_{1}:\bigoplus_{\alpha\in F}E_{\alpha}\rightarrow\bigoplus_{\alpha\in A}E_{\alpha}$ is the inclusion homomorphism. Then for each $a\in I,$ we have that
$g^{\prime}(a)=\pi_{F}(f(a)).$ Since $f(I)\subseteq\bigoplus_{\alpha\in F}E_{\alpha},$ $f(a)\in\bigoplus_{\alpha\in F}E_{\alpha},$  \,$\forall a\in I.$
Thus by Lemma~\ref{lem:(2.3.1)} we have that $\pi_{F}(f(a))=f(a),\,\forall a\in I$
and hence $g^{\prime}(a)=f(a),\,\forall a\in I.$ Since
$\mathcal{L}$ satisfies conditions $\left(\alpha\right)$ and $\left(\mu\right)$, it follows from Theorem~\ref{thm:(2.1.11)} that $\bigoplus_{\alpha\in A}E_{\alpha}$ is $\mathcal{L}$-injective.
\end{proof}

 The following proposition generalizes Proposition 8.13 in \cite[p. 83]{Gol86}.

\begin{prop}\label{prop:(2.3.3)} Let $\mathcal{L}$ satisfy conditions $\left(\alpha\right),\, \left(\mu\right)$
and $\left(\delta\right)$  and let $\{E_{\alpha}\}_{\alpha\in A}$
be any family of $\mathcal{L}$-injective modules, where $A$ is an infinite index set. If
$\bigoplus_{\alpha\in C}E_{\alpha}$ is an $\mathcal{L}$-injective module for any countable subset $C$ of $A,$  then $\bigoplus_{\alpha\in A}E_{\alpha}$ is an $\mathcal{L}$-injective module.
\end{prop}

\begin{proof}
Let $\pi_{\beta}:\bigoplus_{\alpha\in A}E_{\alpha}\rightarrow E_{\beta}$
be the natural projection homomorphism. Assume that $\bigoplus_{\alpha\in A}E_{\alpha}$ is not $\mathcal{L}$-injective, thus by Lemma~\ref{lem:(2.3.2)} there exists $(R,I,f,\bigoplus_{\alpha\in A}E_{\alpha})\in\mathcal{L}$ such that $\textup{supp}(\textup{im}(f))$ is  infinite.
 Since $\textup{supp}(\textup{im}(f))$ is an infinite set,  $\textup{supp}(\textup{im}(f))$ contains a countable infinite subset,  say  $C.$  For  any  $\alpha\in C,$ then $\alpha\in \textup{supp}(\textup{im}(f))$ and this implies that there exists $x\in \textup{im}(f)$ such that
 $x_{\alpha}\neq 0.$  Thus for any $\alpha\in C,$ then $\pi_{\alpha}(\textup{im}(f))\neq 0.$
 Define $\pi_{C}:\bigoplus_{\alpha\in A}E_{\alpha}\rightarrow\bigoplus_{\alpha\in C}E_{\alpha}$
as in  Lemma~\ref{lem:(2.3.1)}.  Note that $C=\textup{supp}(\textup{im}(\pi_{C}\circ f)).$
 Since $(R,I,f,\bigoplus_{\alpha\in A}E_{\alpha})\in\mathcal{L}$ and $\mathcal{L}$ satisfies condition $\left(\gamma\right)$,
$(R,I,\pi_{C}\circ f,\bigoplus_{\alpha\in C}E_{\alpha})\in\mathcal{L}.$ Since $C$ is a countable subset of $A,$ it follows from  the hypothesis
that $\bigoplus_{\alpha\in C}E_{\alpha}$ is $\mathcal{L}$-injective. By Theorem~\ref{thm:(2.1.11)},  there exists an element $y\in\bigoplus_{\alpha\in C}E_{\alpha}$ such that $(\pi_{C}\circ f)(a)=ay,\,\forall a\in I.$  Let $\alpha\in \textup{supp}(\textup{im}(\pi_{C}\circ f)),$
 thus there is $r\in I$ such that $\pi_{\alpha}((\pi_{C}\circ f)(r))\neq 0$. Hence $\pi_{\alpha}(ry)\neq 0$ and this implies that $\pi_{\alpha}(y)\neq 0.$ Thus $\alpha\in \textup{supp}(y)$ and hence $\textup{supp}(\textup{im}(\pi_{C}\circ f))\subseteq \textup{supp}(y).$
 Since $C=\textup{supp}(\textup{im}(\pi_{C}\circ f)),$  $C\subseteq \textup{supp}(y)$ and
this a contradiction, since $\textup{supp}(y)$ is finite (because $y\in\bigoplus_{\alpha\in C}E_{\alpha}$)
and $C$ is infinite. Thus $\bigoplus_{\alpha\in A}E_{\alpha}$ is an $\mathcal{L}$-injective module.
\end{proof}

 By Proposition~\ref{prop:(2.3.3)} and Lemma~\ref{lem:(2.1.9)} we can prove the following
corollary.

\begin{cor}\label{cor:(2.3.4)} Let $\rho_{1}$ and $\rho_{2}$ be any two $P$-filters and let
$\{E_{\alpha}\}_{\alpha\in A}$ be any family of modules, where $A$ is an infinite index set. If $\bigoplus_{\alpha\in C}E_{\alpha}$
is an $\mathcal{L}_{(\rho_{1},\rho_{2})}$-injective module for any countable subset $C$ of $A,$  then $\bigoplus_{\alpha\in A}E_{\alpha}$ is an $\mathcal{L}_{(\rho_{1},\rho_{2})}$-injective module.
\end{cor}

 We can now state the following result, found in \cite[Proposition 8.13, p. 83]{Gol86} as a corollary.

\begin{cor}\label{cor:(2.3.5)} Let $\{E_{\alpha}\}_{\alpha\in A}$ be any family of $\tau$-injective modules, where $A$ is an infinite index set. If $\bigoplus_{\alpha\in C}E_{\alpha}$ is a $\tau$-injective module for any countable subset $C$ of $A,$ then $\bigoplus_{\alpha\in A}E_{\alpha}$ is a $\tau$-injective module.
\end{cor}
\begin{proof}
By taking the two $P$-filters $\rho_{1}=\rho_{\tau}$ and $\rho_{2}=\Re$ and applying Corollary~\ref{cor:(2.3.4)}.
\end{proof}

 Since the class of $\mathcal{L}$-injective modules is closed under isomorphism, when $\mathcal{L}$ satisfies $\left(\gamma\right)$, it follows from Proposition~\ref{prop:(2.3.3)}   we have  the following corollary.

\begin{cor}\label{cor:(2.3.6)} Consider the following three conditions, where $\K$ is a nonempty class of $R$-modules.

$\left(1\right)$ Every direct sum  of  $\mathcal{L}$-injective $R$-modules in $\K$ is $\mathcal{L}$-injective.

$\left(2\right)$ Every countable direct sum of  $\mathcal{L}$-injective $R$-modules in $\K$ is $\mathcal{L}$-injective.

$\left(3\right)$ For any family  $\{E_{i}\}_{i\in N}$ of $\mathcal{L}$-injective $R$-modules in $\K,$   then $\bigoplus_{i\in N}E_{i}$ is $\mathcal{L}$-injective.

Then $\left(1\right)$ implies $\left(2\right)$ and $\left(2\right)$ implies $\left(3\right),$ and if $\mathcal{L}$ satisfies  conditions $\left(\alpha\right),\,\,\left(\mu\right)$ and $\left(\delta\right)$, then $\left(2\right)$ implies $\left(1\right).$ Moreover, if $\mathcal{L}$ satisfies condition $\left(\gamma\right)$, then $\left(3\right)$ implies $\left(2\right).$
\end{cor}

\begin{defn}\label{defn:(2.3.7)} A submodule $N$ of a module $M$ is said to be $strongly\mathcal{L}$-dense in $M$ (shortly, $s$-$\mathcal{L}$-dense), if $(M,N,I_{N},N)\in\mathcal{L},$ where $I_{N}$ is the identity homomorphism from $N$ into $N.$
\end{defn}

The following lemmas are clear.

\begin{lem}\label{lem:(2.3.8)}  If $N\leq K\leq M$ are modules such that $N$ is $s$-$\mathcal{L}$-dense in $M$ and $\mathcal{L}$ satisfies conditions
$\left(\alpha\right)$ and $\left(\beta\right)$, then $K$ is $s$-$\mathcal{L}$-dense in $M.$
\end{lem}

\begin{lem}\label{lem:(2.3.9)} Let $\rho$ be any $P$-filter. Then $(M,N)\in\rho$ if and only if $N$ is $s$-$\mathcal{L}_{(\rho,\infty)}$-dense  in $M.$
\end{lem}

Following \cite[p. 21]{DaZh06}, for any module $M,$ denote by  $H_\K(M)$ the set of left submodules
$N$ of $M$ such that $(M/N)\in \K,$ where $\K$ is any nonempty class of modules (i.e.,  $H_{\K}(M)=\{N\leq M\,\mid\,$
$(M/N)\in \K\}$).  In particular, $H_{\K}(R)=\{I\leq R\,\mid\,(R/I)\in \K\}.$

$\vphantom{}$

 The following theorem is the first main result of this section.

\begin{thm}\label{thm:(2.3.10)} Let $\mathcal{L}$ satisfy  conditions $\left(\alpha\right)$ and $\left(\delta\right)$
 and let $\K$ be any nonempty class of modules closed under isomorphic copies and $\mathcal{L}$-injective hulls. If the direct
sum of any family $\{E_{i}\}_{i\in N}$ of $\mathcal{L}$-injective $R$-modules in $\K$ is $\mathcal{L}$-injective, then every ascending
chain $I_{1}\subseteq I_{2}\subseteq...$ of left ideals of $R$ in
$H_{\K}(R)$ with   $I_{\infty}=\bigcup_{j=1}^{\infty}I_{j}$ $s$-$\mathcal{L}$-dense  in $R,$ terminates.
\end{thm}
\begin{proof}
Let $I_{1}\subseteq I_{2}\subseteq...$ be any ascending chain of left ideals of $R$ in  $H_{\K}(R)$  with $I_{\infty}=\bigcup_{j=1}^{\infty}I_{j}$
being a  $s$-$\mathcal{L}$-dense left ideal in $R.$ Thus $(R/I_{j})\in \K,\,\forall j\in\mathbb{N}.$
 Since $\mathcal{L}$ satisfies conditions $\left(\alpha\right),$\,  $\left(\beta\right)$
and $\left(\gamma\right)$, it follows from \cite[Theorem 1.12, p. 625]{Jir75}  that every $R$-module $M$ has an $\mathcal{L}$-injective hull
which is unique up to $M$-isomorphism. Let $E_{\mathcal{L}}(R/I_{j})$ be the $\mathcal{L}$-injective hull of $R/I_{j},\,\forall j\in\mathbb{N}.$
 Since $\K$ closed under $\mathcal{L}$-injective hulls, $E_{\mathcal{L}}(R/I_{j})\in \K,\,\forall j\in\mathbb{N}.$
 Define $f:I_{\infty}=\bigcup_{j=1}^{\infty}I_{j}\rightarrow\bigoplus_{j=1}^{\infty}E_{\mathcal{L}}(R/I_{j})$
by $f(r)=(r+I_{j})_{j\in\mathbb{N}},$  \,for $r\in I_{\infty}.$ Note that
 $f$ is a well-defined mapping:  for any  $r\in I_{\infty}$, let
$n$ be the smallest positive integer such that $r\in I_{n}.$ Since
$I_{n}\subseteq I_{n+k},\,\forall k\in\mathbb{N},$ we have $r\in I_{n+k},\,\forall k\in\mathbb{N}$
and so $r+I_{n+k}=0,\,\forall k\in\mathbb{N}.$ Thus $(r+I_{j})_{j\in\mathbb{N}}=(r+I_{1},r+I_{2},...,r+I_{n-1},0,0,...)\in\bigoplus_{j=1}^{\infty}E_{\mathcal{L}}(R/I_{j}).$
Thus $f(I)\subseteq\bigoplus_{j=1}^{\infty}E_{\mathcal{L}}(R/I_{j})$
and hence $f$ is a well-defined mapping. It is clear that $f$ is
a homomorphism. Since $I_{\infty}$ is a $s$-$\mathcal{L}$-dense
left ideal in $R,$ it follows that $(R,I_{\infty},I_{I_{\infty}},I_{\infty})\in\mathcal{L}.$
Since $\mathcal{L}$ satisfies condition $\left(\delta\right)$,
$(R,I_{\infty},f,\bigoplus_{j=1}^{\infty}E_{\mathcal{L}}(R/I_{j}))\in\mathcal{L}.$
Since $E_{\mathcal{L}}(R/I_{j})$ is an $\mathcal{L}$-injective $R$-module
in $\K,\,\forall j\in\mathbb{N},$ it follows from the hypothesis  that
$\bigoplus_{j=1}^{\infty}E_{\mathcal{L}}(R/I_{j})$ is an $\mathcal{L}$-injective
$R$-module. Thus by Theorem~\ref{thm:(2.1.11)}, there exists an element $x\in\bigoplus_{j=1}^{\infty}E_{\mathcal{L}}(R/I_{j})$
such that $f(r)=rx,\,\forall r\in I_{\infty}.$ Since $x\in\bigoplus_{j=1}^{\infty}E_{\mathcal{L}}(R/I_{j})$,
 $x=(x_{1},x_{2},...,x_{n},0,0,...),$  \,for some $n\in\mathbb{N}$
 and hence  $(r+I_{j})_{j\in N}=(rx_{1},rx_{2},...,rx_{n},0,0,...)$ and
this implies that $r+I_{n+k}=0,\,\forall k\geq1$ and \,$\forall r\in I_{\infty},$
 Thus $r\in I_{n+k},\,\forall k\geq1$ and \, $\forall r\in I_{\infty}$
and so $I_{\infty}=\bigcup_{j=1}^{\infty}I_{j}\subseteq I_{n+k},\,\forall k\geq1.$
Since $I_{n+k}\subseteq I_{\infty},\,I_{\infty}=I_{n+k},\forall k\geq 1,$
 $I_{t}=I_{t+j},\,\forall j\in\mathbb{N}.$ Therefore
the ascending chain $I_{1}\subseteq I_{2}\subseteq...$  terminates.
\end{proof}

 Now we will state the condition $\left(I\right)$  on $\mathcal{L}$
as follows:

$\left(I\right):\,(R,J,f,Q)\in\mathcal{L}$ implies that $J$ is $s$-$\mathcal{L}$-dense in  $R.$  That is, $(R,J,f,Q)\in\mathcal{L}$ implies $(R,J,I_J,J)\in\mathcal{L}$.

\begin{prop}\label{prop:(2.3.11)} Consider the following two conditions, where $\K$ is a nonempty class of $R$-modules.

$\left(1\right)$ Every ascending chain $I_{1}\subseteq I_{2}\subseteq...$ of left ideals of $R$ in $H_{\K}(R)$ with  $I_{\infty}=\bigcup_{j=1}^{\infty}I_{j}$  $s$-$\mathcal{L}$-$dense$  in $R,$ terminates.

$\left(2\right)$ The following conditions hold:

$\quad \left(a\right)$ $H_{\K}(R)$ has $ACC$ on $s$-$\mathcal{L}$-dense left ideals in $R;$

$\quad\left(b\right)$ for every ascending chain $I_{1}\subseteq I_{2}\subseteq...$
of left ideals of $R$ in $H_{\K}(R)$ with $I_{\infty}=\bigcup_{j=1}^{\infty}I_{j}$
$s$-$\mathcal{L}$-dense  in $R,$ there exists a positive integer $n$ such that $I_{n}$ is $s$-$\mathcal{L}$-dense in $R.$

If $\mathcal{L}$ satisfies conditions $\left(\alpha\right)$ and $\left(\beta\right)$, then $\left(1\right)$ and $\left(2\right)$ are equivalent.
\end{prop}
\begin{proof} Clearly $\left(1\right)\Rightarrow\left(2b\right)$.

$\hspace{1em}\left(1\right)\Rightarrow\left(2a\right)$ Assume that $\mathcal{L}$ satisfies conditions $\left(\alpha\right)$
and $\left(\beta\right)$  and let $I_{1}\subseteq I_{2}\subseteq...$ be any ascending chain of $s$-$\mathcal{L}$-dense left ideals
of $R$ in $H_{\K}(R).$  Since $I_{1}\subseteq I_{\infty}=\bigcup_{j=1}^{\infty}I_{j},$
it follows from  Lemma~\ref{lem:(2.3.8)}   that  $I_{\infty}$ is $s$-$\mathcal{L}$-dense in $R.$ By hypothesis,  the chain $I_{1}\subseteq I_{2}\subseteq...$ terminates. Thus $H_{\K}(R)$  has $ACC$ on $s$-$\mathcal{L}$-dense left ideals in $R.$

$\left(2\right)\Rightarrow\left(1\right)$ Assume that $\mathcal{L}$ satisfies conditions $\left(\alpha\right)$ and $\left(\beta\right)$
and let $I_{1}\subseteq I_{2}\subseteq...$ be any ascending chain of left ideals of $R$ in  $H_{\K}(R)$,  with $I_{\infty}=\bigcup_{j=1}^{\infty}I_{j}$ $s$-$\mathcal{L}$-dense in $R.$  By $\left(2b\right)$, there exists a positive integer $n$ such that $I_{n}$ is $s$-$\mathcal{L}$-dense in $R.$ Consider the following ascending chain $I_{n}\subseteq I_{n+1}\subseteq...$ of left ideals of $R.$ Since $\mathcal{L}$ satisfies  conditions $\left(\alpha\right)$ and $\left(\beta\right)$,  it follows from Lemma~\ref{lem:(2.3.8)}   that $I_{n+j}$ is $s$-$\mathcal{L}$-dense left ideal in $R,$  $\forall j\in\mathbb{N}.$ By $\left(2a\right)$, there exists a positive integer $t\geq n$ such that $I_{t}=I_{t+j},\,\forall j\in\mathbb{N}.$ Thus the ascending chain $I_{1}\subseteq I_{2}\subseteq...$ of left ideals  terminates.
\end{proof}

 Now we will give the second main result of this section.

\begin{thm}\label{thm:(2.3.12)} Let $\mathcal{L}$ satisfy conditions $\left(\alpha\right), \,\, \left(\mu\right), \,\, \left(\delta\right)$
and $\left(I\right)$  and let $\K$ be any nonempty class of modules closed under isomorphic copies and submodules. If every ascending chain $J_{1}\subseteq J_{2}\subseteq...$ of left ideals of $R$ such that $(J_{i+1}/J_{i})\in \K,\, \forall i\in\mathbb{N}$
and $J_{\infty}=\bigcup_{i=1}^{\infty}J_{i}$ \, $s$-$\mathcal{L}$-dense in $R$ terminates, then every direct sum of $\mathcal{L}$-injective
modules in $\K$ is $\mathcal{L}$-injective.
\end{thm}
\begin{proof} Let $\{E_{i}\}_{i\in\mathbb{N}}$ be any family of $\mathcal{L}$-injective modules in $\K$ and let $(R,J,f,\bigoplus_{i\in\mathbb{N}}E_{i})\in\mathcal{L}.$ For any $n\in\mathbb{N},$ put $J_{n}=\{x\in J\;\mid  f(x)\in\bigoplus_{i=1}^{n}E_{i}\}=f^{-1}(\bigoplus_{i=1}^{n}E_{i}).$ It is clear that $J_{1}\subseteq J_{2}\subseteq...$.
Also, we have that $J_{\infty}=\bigcup_{n\in\mathbb{N}}J_{n}=\bigcup_{n\in\mathbb{N}}(f^{-1}(\bigoplus_{i=1}^{n}E_{i}))=f^{-1}(\bigcup_{n\in\mathbb{N}}(\bigoplus_{i=1}^{n}
E_{i})=f^{-1}(\bigoplus_{i=1}^{\infty}E_{i}).$ Since $(R,J,f,\bigoplus_{i\in\mathbb{N}}E_{i})\in\mathcal{L}$
and $\mathcal{L}$ satisfies  condition $\left(I\right),\, J=\bigcup_{i\in\mathbb{N}}J_{i}$ is $s$-$\mathcal{L}$-dense
in $R.$  For all $n\in\mathbb{N},$  define $\alpha_{n}:J_{n+1}/J_{n}\rightarrow$$\bigoplus_{i=1}^{n+1}E_{i}/\bigoplus_{i=1}^{n}E_{i}$
by $\alpha_{n}(x+J_{n})=f(x)+(\bigoplus_{i=1}^{n}E_{i}),\,\forall x\in I_{n+1}.$
$\alpha_{n}$  is a well-defined monomorphism,
 since $J_{n}=f^{-1}(\bigoplus_{i=1}^{n}E_{i})$. Since $(\bigoplus_{i=1}^{n+1}E_{i}/\bigoplus_{i=1}^{n}E_{i})\simeq E_{n+1}\in \K$
and $\K$ closed under isomorphic copies,  $(\bigoplus_{i=1}^{n+1}E_{i}/\bigoplus_{i=1}^{n}E_{i})$
$\in \K.$
Since $\textup{im}(\alpha_{n})\leq(\bigoplus_{i=1}^{n+1}E_{i}/\bigoplus_{i=1}^{n}E_{i})\in \K,$
   and $\K$ closed under submodules,  $\textup{im}(\alpha_{n})\in \K.$ Since $(J_{n+1}/J_{n})\simeq \textup{im}(\alpha_{n})$
and $\K$ closed under isomorphic copies,  $(J_{n+1}/J_{n})\in \K.$ Thus we have the following ascending chain
$J_{1}\subseteq J_{2}\subseteq...$ of left ideals of $R$ such that
$(J_{i+1}/J_{i})\in \K,\,\forall i\in\mathbb{N}$ and $J_{\infty}=\bigcup_{i=1}^{\infty}J_{i}$
is $s$-$\mathcal{L}$-dense in $R$. By hypothesis, there exists a positive integer $n$ such that $J_{n}=J_{n+i},\,\forall i\in\mathbb{N}.$  Thus $J=J_{\infty}=\bigcup_{i=1}^{\infty}J_{i}=J_{n}.$ This implies that $f(J)\subseteq\bigoplus_{i=1}^{n}E_{i}.$
 Thus $\textup{supp}(\textup{\textup{im}}(f))$ is finite and hence $\mathcal{L}$ satisfies condition $\left(F\right)$
for a family $\{E_{i}\}_{i\in\mathbb{N}}.$  Thus by Lemma~\ref{lem:(2.3.2)} we
have that $\bigoplus_{i\in\mathbb{N}}E_{i}$ is an $\mathcal{L}$-injective
module.  Thus for any family $\{E_{i}\}_{i\in N}$ of $\mathcal{L}$-injective
$R$-modules in $\K,$ we have $\bigoplus_{i\in N}E_{i}$ is $\mathcal{L}$-injective.
 Since $\mathcal{L}$ satisfies conditions $\left(\alpha\right),\,\,\left(\mu\right)$
and $\left(\delta\right)$, it follows from  Corollary~\ref{cor:(2.3.6)},   that every direct sum of $\mathcal{L}$-injective
modules in $\K$ is $\mathcal{L}$-injective.
\end{proof}

A nonempty class $\K$ of modules is said to be a natural class if it is closed under submodules, arbitrary direct sums and injective
hulls \cite{Dau91}.  Examples of natural classes include $R\textup{-Mod},$ any hereditary torsionfree classes and stable hereditary torsion classes.

$\vphantom{}$

 We can now state the following result,  found in \cite[p. 643]{PaZh94} as a corollary.

\begin{cor}\label{cor:(2.3.13)} Let $\K$ be a natural class of modules closed under isomorphic copies. Then the following statements are equivalent:

$\left(1\right)$ every direct sum of injective modules in $\K$ is injective;

$\left(2\right)$ $H_{\K}(R)$ has $ACC.$
\end{cor}
\begin{proof}
$\left(1\right)\Rightarrow\left(2\right)$ By taking $\mathcal{L}=\M$
and applying Lemma~\ref{lem:(2.1.9)}, Lemma~\ref{lem:(2.3.9)} and Theorem~\ref{thm:(2.3.10)}.

$\left(2\right)\Rightarrow\left(1\right)$ By taking $\mathcal{L}=\M$
and applying \cite[Lemma 7, p. 637]{PaZh94} and Theorem~\ref{thm:(2.3.12)}.
\end{proof}

\begin{cor}\label{cor:(2.3.14)} Let $\rho$ be any $P$-filter and let $\K$ be any nonempty class of
modules closed under isomorphic copies and submodules. If every ascending chain $J_{1}\subseteq J_{2}\subseteq...$ of left ideals of $R$ such
that $(J_{i+1}/J_{i})\in \K, \,\forall i\in\mathbb{N}$  and  $J_{\infty}=\bigcup_{i=1}^{\infty}J_{i}$  is $s$-$\mathcal{L}_{(\rho,\infty)}$-dense
in $R$ terminates, then every direct sum of $\mathcal{L}_{(\rho,\infty)}$-injective modules in $\K$ is $\mathcal{L}_{(\rho,\infty)}$-injective.
\end{cor}
\begin{proof}
By Lemma~\ref{lem:(2.1.9)},  Lemma~\ref{lem:(2.3.9)} and Theorem~\ref{thm:(2.3.12)}.
\end{proof}

Let $\tau$ be a hereditary torsion theory. A nonempty class $\K$ of modules is said to be $\tau$-natural class if $\K$ closed under submodules, isomorphic copies, arbitrary direct sums and $\tau$-injective hulls \cite[p. 163]{Cri04}.

\begin{cor}\label{cor:(2.3.15)}\textbf{(\cite[Proposition 5.3.5, p. 165]{Cri04})} Let $\K$ be a $\tau$-natural and suppose that every
ascending chain $J_{1}\subseteq J_{2}\subseteq...$ of left ideals of $R$ such that $(J_{i+1}/J_{i})\in \K,\,\forall i\in\mathbb{N}$ and  $J_{\infty}=\bigcup_{i=1}^{\infty}J_{i}$ is $\tau$-dense in $R$ terminates. Then every direct sum of $\tau$-injective
modules in $\K$ is $\tau$-injective.
\end{cor}
\begin{proof}
 Take $\rho=\rho_\tau$ and apply Corollary~\ref{cor:(2.3.14)}.
 \end{proof}

 The following corollary, in which we give conditions under which the class of $\mathcal{L}$-injective
modules is closed under direct sums,  is one of the main aims of this section.

\begin{cor}\label{cor:(2.3.17)} Consider the following three conditions:

$\left(1\right)$ the class of $\mathcal{L}$-injective $R$-modules is closed under direct sums;

$\left(2\right)$ every ascending chain $I_{1}\subseteq I_{2}\subseteq...$ of left ideals of $R$ with  $I_{\infty}=\bigcup_{j=1}^{\infty}I_{j}$  $s$-$\mathcal{L}$-dense  in $R,$ terminates;

$\left(3\right)$ the following conditions hold:

$\quad$$\left(a\right)$ every ascending chain $I_{1}\subseteq I_{2}\subseteq...$ of $s$-$\mathcal{L}$-dense left ideals of $R,$ terminates;

$\quad$$\left(b\right)$ for every ascending chain $I_{1}\subseteq I_{2}\subseteq...$
of left ideals of $R$ with  $I_{\infty}=\bigcup_{j=1}^{\infty}I_{j}$  $s$-$\mathcal{L}$-dense  in $R,$ there exists a
positive integer $n$ such that $I_{n}$ is $s$-$\mathcal{L}$-dense in $R.$

If $\mathcal{L}$ satisfies conditions $\left(\alpha\right)$ and $\left(\delta\right)$,
 then $\left(1\right)$ implies $\left(2\right).$ Also, $\left(2\right)$
implies $\left(3b\right)$  and if $\mathcal{L}$ satisfies conditions $\left(\alpha\right)$ and $\left(\beta\right)$, then $\left(2\right)$ implies $\left(3a\right)$.
 Moreover, if $\mathcal{L}$ satisfies conditions $\left(\alpha\right),\,\,  \left(\mu\right),\,\,   \left(\delta\right)$
and $\left(I\right)$, then all above three conditions are equivalent.
\end{cor}
\begin{proof} $\left(1\right)\Rightarrow\left(2\right)$ Let $\mathcal{L}$ satisfy conditions $\left(\alpha\right)$ and $\left(\delta\right)$. Take $\K=R\textup{-Mod}$ and apply Theorem~\ref{thm:(2.3.10)}.

$\left(2\right)\Rightarrow\left(3b\right)$ Take $\K=R\textup{-Mod}$ and apply Proposition~\ref{prop:(2.3.11)}.

 $\left(2\right)\Rightarrow\left(3a\right)$ Let $\mathcal{L}$ satisfy conditions $\left(\alpha\right)$ and $\left(\beta\right)$. Take
$\K=R\textup{-Mod}$ and apply Proposition~\ref{prop:(2.3.11)}.

 $\left(3\right)\Rightarrow\left(1\right)$ Let $\mathcal{L}$ satisfy conditions $\left(\alpha\right),\, \,\left(\mu\right),\, \,\left(\delta\right)$ and $\left(I\right)$. Take $\K=R\textup{-Mod}$. By Proposition~\ref{prop:(2.3.11)}, we have that every ascending chain
 $I_{1}\subseteq I_{2}\subseteq...$ of left ideals of $R$ with $I_{\infty}=\bigcup_{j=1}^{\infty}I_{j}$
 $s$-$\mathcal{L}$-dense  in $R,$ terminates. Thus every ascending chain $I_{1}\subseteq I_{2}\subseteq...$
of left ideals of $R$ such that $(I_{j+1}/I_{j})\in \K$,
  $\,\forall j\in\mathbb{N}$  and $I_{\infty}=\bigcup_{j=1}^{\infty}I_{j}$
 $s$-$\mathcal{L}$-dense  in $R,$ terminates.
 Since $\K$ is closed under isomorphic copies and submodules,  we have from Theorem~\ref{thm:(2.3.12)}
 that every direct sum of $\mathcal{L}$-injective modules in $\K$ is $\mathcal{L}$-injective. Thus the class of $\mathcal{L}$-injective
$R$-modules is closed under direct sums.
\end{proof}

\begin{cor}\label{cor:(2.3.18)} Let $\rho$ be any $P$-$filter.$ Then the following  statements are equivalent.

$\left(1\right)$ The class of $\mathcal{L}_{(\rho,\rho_{\infty})}$-injective
$R$-modules is closed under direct sums.

$\left(2\right)$ Every ascending chain $I_{1}\subseteq I_{2}\subseteq...$ of left ideals of $R$ with $I_{\infty}=\bigcup_{j=1}^{\infty}I_{j}$
 $s$-$\mathcal{L}_{(\rho,\rho_{\infty})}$-dense  in $R,$ terminates.

$\left(3\right)$ The following conditions hold.

$\quad$$\left(a\right)$ Every ascending chain $I_{1}\subseteq I_{2}\subseteq...$ of $s$-$\mathcal{L}_{(\rho,\rho_{\infty})}$-dense left ideals of $R,$
terminates.

$\quad$$\left(b\right)$ For every ascending chain $I_{1}\subseteq I_{2}\subseteq...$ of left ideals of $R$ with  $I_{\infty}=\bigcup_{j=1}^{\infty}I_{j}$  \, $s$-$\mathcal{L}_{(\rho,\rho_{\infty})}$-dense in $R,$ there exists a positive integer $n$ such that $I_{n}$ is $s$-$\mathcal{L}_{(\rho,\rho_{\infty})}$-dense in $R.$

$\left(4\right)$ For any family $\{E_{i}\}_{i\in N}$ of $\mathcal{L}_{(\rho,\rho_{\infty})}$-injective
$R$-modules,  $\bigoplus_{i\in N}E_{i}$ is $\mathcal{L}_{(\rho,\rho_{\infty})}$-injective.
\end{cor}
\begin{proof} By Lemma~\ref{lem:(2.1.9)} and Lemma~\ref{lem:(2.3.9)} we have that $\mathcal{L}_{(\rho,\rho_{\infty})}$ satisfies conditions $\left(\alpha\right),\,\left(\mu\right),\,\left(\delta\right)$ and $\left(I\right)$. Thus by Corollary~\ref{cor:(2.3.17)} and Corollary~\ref{cor:(2.3.6)} we have the equivalence of above four statements.
\end{proof}

\begin{cor}\label{cor:(2.3.19)}\textbf{(\cite[Theorem 2.3.8, p. 73]{Cri04}) }The following statements are equivalent:

$\left(1\right)$ $R$ has ACC on $\tau$-dense left ideals and $\tau$ is Noetherian;

$\left(2\right)$ the class of $\tau$-injective $R$-modules is closed under direct sums;

$\left(3\right)$ the class of $\tau$-injective $R$-modules is closed under countable direct sums.
\end{cor}
\begin{proof} Take $\rho=\rho_{\tau}$ and apply Corollary~\ref{cor:(2.3.18)}.
\end{proof}

\section{$\Sigma$-$\mathcal{L}$-Injective Modules }

 Carl Faith in \cite{Fai66} introduced the concepts of $\sum$-injectivity and countably $\sum$-injectivity as follows. An injective module
$E$ is said to be  $\sum$-injective if $E^{(A)}$ is injective for any index set $A;\,\,E$ is said to be countably $\sum$-injective in case
$E^{(C)}$ is injective for any countable index set $C.$ Faith in \cite{Fai66} proved that an injective $R$-module $E$ is $\sum$-injective
if and only if $R$ satisfies $ACC$ on the $E$-annihilator left ideals if and only if $E$ is countably $\sum$-injective. S. Charalambides
in \cite{Cha06} introduced the concept of $\sum$-$\tau$-injectivity and generalized Faith's result.

$\vphantom{}$

In this section, we introduce the concept of $\sum$-$\mathcal{L}$-injectivity as a general case of $\sum$-injectivity and $\sum$-$\tau$-injectivity
and prove the result (Theorem~\ref{thm:(2.4.4)}) in which we generalize Faith's result \cite[Proposition 3, p. 184]{Fai66}
and \cite[Theorem 16.16, p. 98]{Cha06}.

$\vphantom{}$

We start this section with the following definition of a $\sum$-$\mathcal{L}$-injective module.

\begin{defn}\label{defn:(2.4.1)} Let $E$ be an $\mathcal{L}$-injective module. We say that $E$ is $\sum$-$\mathcal{L}$-injective if $E^{(A)}$ is $\mathcal{L}$-injective for any index set $A.$ On other hand, if $E^{(C)}$ is $\mathcal{L}$-injective for any countable index set $C,$ we say that $E$ is countably $\sum$-$\mathcal{L}$-injective.
\end{defn}

The following corollary is a special case of Corollary~\ref{cor:(2.3.6)}, by taking $\K=\{E\}.$

\begin{cor}\label{cor:(2.4.2)} Consider the following conditions.

$\left(1\right)$ $E$ is $\sum$-$\mathcal{L}$-injective.

$\left(2\right)$ $E$ is countably $\sum$-$\mathcal{L}$-injective.

$\left(3\right)$ $E^{(\mathbb{N})}$ is $\mathcal{L}$-injective.

Then: $\left(1\right)$ implies $\left(2\right)$ and $\left(2\right)$ implies $\left(3\right).$  If $\mathcal{L}$ satisfies  conditions $\left(\alpha\right), \,\,\left(\mu\right)$ and $\left(\delta\right)$, then $\left(2\right)$ implies $\left(1\right).$  \,Moreover, if $\mathcal{L}$ satisfies condition $\left(\gamma\right)$, then $\left(3\right)$ implies $\left(2\right).$
\end{cor}

The following corollary is immediately from Lemma~\ref{lem:(2.1.9)} and Corollary~\ref{cor:(2.4.2)}.

\begin{cor}\label{cor:(2.4.3)} Let $\rho_{1}$ and $\rho_{2}$ be any two $P$-filters.  Then the
following conditions are equivalent for a module $E$.

$\left(1\right)$ $E$ is $\sum$-$\mathcal{L}_{(\rho_{1},\rho_{2})}$-injective.

$\left(2\right)$ $E$ is countably $\sum$-$\mathcal{L}_{(\rho_{1},\rho_{2})}$-injective.

$\left(3\right)$ $E^{(\mathbb{N})}$ is  $\mathcal{L}_{(\rho_{1},\rho_{2})}$-injective.
\end{cor}

\,\,Let \, $E$ be a module. A left ideal $I$ of $R$ is said to be an \, $E$-annihilator if there is \, $N\subseteq E$ \,such that $I=(0:N)=\{r\in R\mid rN=0\}$ (i.e., $I$ is the annihilator of a subset of $E$).

$\vphantom{}$

The following theorem is the main result of this section in which we generalize \cite[Theorem 16.16, p. 98]{Cha06}
and \cite[Proposition 3, p. 184]{Fai66}.

\begin{thm}\label{thm:(2.4.4)} Consider the following three conditions for an $\mathcal{L}$-injective module $E$:

$\left(1\right)$ $E$ is countably $\sum$-$\mathcal{L}$-injective;

$\left(2\right)$ every ascending chain $I_{1}\subseteq I_{2}\subseteq...$
of $E$-annihilators in $R$ with $I_{\infty}=\bigcup_{j=1}^{\infty}I_{j}$
  $s$-$\mathcal{L}$-dense  in $R,$ terminates;

$\left(3\right)$ The following conditions hold.

$\quad$$\left(a\right)$ Every ascending chain $I_{1}\subseteq I_{2}\subseteq...$
of $E$-annihilators in $R$ with $I_{j}$ being $s$-$\mathcal{L}$-dense in $R,\,\forall j\in\mathbb{N},$  terminates.

$\quad$$\left(b\right)$ For every ascending chain $I_{1}\subseteq I_{2}\subseteq...$
of $E$-annihilators in $R$ with  $I_{\infty}=\bigcup_{j=1}^{\infty}I_{j}$
 $s$-$\mathcal{L}$-dense in $R,$  there exists a
positive integer $n$ such that $I_{n}$ is $s$-$\mathcal{L}$-dense
in $R.$

Then: if $\mathcal{L}$ satisfies condition $\left(\delta\right)$,
then $\left(1\right)$ implies $\left(2\right).$  Also, $\left(2\right)$
implies $\left(3b\right)$ and if $\mathcal{L}$ satisfies conditions $\left(\alpha\right)$ and $\left(\beta\right)$,
then $\left(2\right)$ implies $\left(3a\right)$.
Moreover, if $\mathcal{L}$ satisfies conditions $\left(\alpha\right),\,\left(\mu\right),\,\left(\beta\right)$
and $\left(I\right)$, then $\left(3\right)$ implies  $\left(1\right)$.
\end{thm}
\begin{proof} $\left(1\right)\Rightarrow\left(2\right)$ Let $\mathcal{L}$ satisfy condition $\left(\delta\right)$.
Assume that $\left(2\right)$ does not hold. Then there exist
$E$-annihilators $I_{1},I_{2},...$ in $R$ such that $I_{1}\subsetneqq I_{2}\subsetneqq...$
and $I_{\infty}=\bigcup_{j=1}^{\infty}I_{j}$ is $s$-$\mathcal{L}$-dense in $R.$
Hence we have the following descending chain $r_{E}(I_{1})\supsetneqq r_{E}(I_{2})\supsetneqq...$. For every $n\in\mathbb{N},$  choose $x_{n}\in r_{E}(I_{n})-r_{E}(I_{n+1}),$ thus $x=(x_{n})_{n\in\mathbb{N}}\in E^{\mathbb{N}}.$   Define $f:I_{\infty}\rightarrow E^{\mathbb{N}}$
by $f(a)=ax,\,\forall a\in I_{\infty}.$ It is clear that $f$ is a homomorphism. For a fixed $a\in I_{\infty}$ let $n$ be the
smallest positive integer such that $a\in I_{n}.$ Then, for every
$k\geq 0,\,a\in I_{n}\subseteq I_{n+k}.$ Since $x_{n+k}\in r_{E}(I_{n+k}),$
 $ax_{n+k}=0,\,\forall k\geq 0.$ Hence $ax\in E^{(\mathbb{N})}.$
Thus $f$ is a  homomorphism from $I_{\infty}$ into  $E^{(\mathbb{N})}.$
Since $I_{\infty}$ is $s$-$\mathcal{L}$-dense  in $R,$   $(R,I_{\infty},I_{I_{\infty}},I_{\infty})\in\mathcal{L}.$
 Since $\mathcal{L}$ satisfies  condition $\left(\delta\right)$,
$(R,I_{\infty},f,E^{(\mathbb{N})})\in\mathcal{L}.$ Since  $E^{(\mathbb{N})}$ is   $\mathcal{L}$-injective, it follows from Theorem~\ref{thm:(2.1.11)},   that there exists an element $y\in E^{(\mathbb{N})}$ such that $f(a)=ay,\,\forall a\in I_{\infty}.$
Since $y\in E^{(\mathbb{N})},$  $y=(y_{1},y_{2},...,y_{t},0,0,...),$ for some $t\in\mathbb{N}.$ Since $ax=f(a)=ay,\,\forall a\in I_{\infty},$
  $(ax_{1},ax_{2},...)=(ay_{1},ay_{2},...,ay_{t},0,0,...)$ and this implies that $ax_{t+1}=0,\,\forall a\in I_{\infty}$ and hence
  $x_{t+1}\in r_{E}(I_{\infty}).$ Since $I_{t+2}\subsetneqq I_{\infty},$   $r_{E}(I_{\infty})\subseteq r_{E}(I_{t+2})$ and so
 $x_{t+1}\in r_{E}(I_{t+2})$. This contradicts the fact that $x_{t+1}\in r_{E}(I_{t+1})-r_{E}(I_{t+2}).$

$\left(2\right)\Rightarrow\left(3b\right)$ Let $I_{1}\subseteq I_{2}\subseteq...$ be any ascending chain of $E$-annihilators
in $R$ with $I_{\infty}=\bigcup_{j=1}^{\infty}I_{j}$ \, $s$-$\mathcal{L}$-dense in $R.$
By hypothesis, there exists a positive integer $n$ such that $I_{n}=I_{n+k},\,\forall k\in\mathbb{N}$ and so
$I_{n}=I_{\infty}.$  Hence $I_{n}$ is $s$-$\mathcal{L}$-dense in $R.$

$\left(2\right)\Rightarrow\left(3a\right)$ Let $\mathcal{L}$ satisfy conditions $\left(\alpha\right)$ and $\left(\beta\right)$
 and let  $I_{1}\subseteq I_{2}\subseteq...$ \, be any ascending chain of $E$-annihilators in $R,$  such
that   $I_{j}$ are $s$-$\mathcal{L}$-dense left ideals of $R.$
 Since $I_{1}\subseteq I_{\infty}$ and $\mathcal{L}$ satisfies conditions $\left(\alpha\right)$ and $\left(\beta\right)$, we have from Lemma~\ref{lem:(2.3.8)}  that $I_{\infty}$ is a $s$-$\mathcal{L}$-dense left ideal of $R.$ By hypothesis,
   the chain $I_{1}\subseteq I_{2}\subseteq...$ terminates.

 $\left(3\right)\Rightarrow\left(1\right)$ Let $\mathcal{L}$ satisfy conditions $\left(\alpha\right),\,\left(\mu\right),\,\left(\beta\right)$
and $\left(I\right)$ and let $(R,J,f,E^{(\mathbb{N})})\in\mathcal{L}.$
Since $E$ is $\mathcal{L}$-injective,  we have from Lemma~\ref{lem:(2.1.3)}  that $E^{\mathbb{N}}$ is $\mathcal{L}$-injective. Since $E^{(\mathbb{N})}$
is a submodule of $E^{\mathbb{N}},$   $g=i\circ f:J\rightarrow E^{\mathbb{N}}$ is a homomorphism, where
$i:E^{(\mathbb{N})}\rightarrow E^{\mathbb{N}}$ is the inclusion homomorphism. Since $\mathcal{L}$ satisfies condition $\left(\beta\right)$,  $(R,J,i\circ f,E^{\mathbb{N}})\in\mathcal{L}.$ Thus by Theorem~\ref{thm:(2.1.11)}, there is an element $x=(x_{1},x_{2},...)\in E^{\mathbb{N}}$
such that $g(a)=ax,\,\forall a\in J.$   Thus $f(a)=g(a)=ax,\,\forall a\in J.$   Let $X=\{x_{1},x_{2},...\}$ and
$X_{k}=X\setminus \{x_{1},x_{2},...,x_{k}\}=\{x_{k+1},x_{k+2},...\},$ \, for all $k\geq1.$ Thus we have the following descending chain of
subsets of $X:\,\,\,X\supseteq X_{1}\supseteq X_{2}\supseteq...$; this  yields an ascending chain of $E$-annihilators in
$R$: \,\,$l_{R}(X)\subseteq l_{R}(X_{1})\subseteq l_{R}(X_{2})\subseteq...$.
Let $J_{k+1}=l_{R}(X_{k}),$ \, for all $k\geq 0,$ where $X_{0}=X$ and $J_{\infty}=\bigcup_{i=1}^{\infty}J_{i}.$
Since $f(J)\subseteq E^{(\mathbb{N})},$ \, for any $a\in J,$  it follows that either $ax_{k}=0,\,\forall k\in\mathbb{N},$
or there is a largest integer $n\in\mathbb{N}$ such that $ax_{n}\neq 0.$
If there is a largest integer $n\in\mathbb{N}$ such that $ax_{n}\neq 0,$ then $ax_{n+k}=0,\,\forall k\geq 1.$
 Therefore, $a\in l_{R}(X_{n})=J_{n+1}\subseteq J_{\infty}.$
Thus for any $a\in J,$ we have $a\in J_{\infty},$ and this implies that $J\subseteq J_{\infty}.$ Since $(R,J,f,E^{(\mathbb{N})})\in\mathcal{L}$
and $\mathcal{L}$ satisfies  condition $\left(I\right)$,   $J$ is $s$-$\mathcal{L}$-dense left ideal in $R.$ Since $J\subseteq J_{\infty}$
and $\mathcal{L}$ satisfies conditions $\left(\alpha\right)$ and $\left(\beta\right)$, we have from   Lemma~\ref{lem:(2.3.8)}  that $J_{\infty}$ is $s$-$\mathcal{L}$-dense left ideal in $R.$ Thus we have the following ascending chain $J_{1}\subseteq J_{2}\subseteq...$
of $E$-annihilators in $R$ such that $J_{\infty}$ is $s$-$\mathcal{L}$-dense left ideal in $R.$
By applying condition $\left(3b\right),$ there is $s\in\mathbb{N}$ such that $J_{s}$ is $s$-$\mathcal{L}$-dense
left ideal in $R.$ Since $J_{s}\subseteq J_{s+k},\,\forall k\in\mathbb{N}$ and $\mathcal{L}$ satisfies conditions $\left(\alpha\right)$ and $\left(\beta\right)$, it follows Lemma~\ref{lem:(2.3.8)} implies that $J_{s+k}$ is $s$-$\mathcal{L}$-dense left ideal in $R,\,\forall k\in\mathbb{N}.$ Thus we have the following ascending chain $J_{s}\subseteq J_{s+1}\subseteq...$ of $E$-annihilators
in $R$ such that $J_{s+k}$ is $s$-$\mathcal{L}$-dense left ideal in $R,\,\forall k\in\mathbb{N}.$
By applying condition $\left(3a\right),$ the chain $J_{s}\subseteq J_{s+1}\subseteq...$ becomes stationary
at a left ideal of $R,$ say $J_{t}=l_{R}(X_{t-1})$ and so $J_{t}=J_{\infty}.$
Thus for any $a\in J,$ we have $ax_{t+k}=0,\,\forall k\geq 0$ and
then $a(0,0,...,0,x_{t},x_{t+1},...)=0.$ Take $y=(x_{1},x_{2},...,x_{t-1},0,0,...).$
It is clear that $y\in E^{(\mathbb{N})}$ and for any $a\in J,$ then $f(a)=ax=ax-a(0,0,...,0,x_{t},x_{t+1},0,0,...)=a(x_{1},x_{2},...,x_{t-1},0,0,...)=ay.$
Thus for every $(R,J,f,E^{(\mathbb{N})})\in\mathcal{L},$ there exists an element $y\in E^{(\mathbb{N})}$ such that $f(a)=ay,\,\forall a\in J.$
Since $\mathcal{L}$ satisfies conditions $\left(\alpha\right)$ and $\left(\mu\right)$,   $E^{(\mathbb{N})}$ is $\mathcal{L}$-injective,
by Theorem~\ref{thm:(2.1.11)}. Since $\mathcal{L}$ satisfies condition $\left(\gamma\right)$,  $E$ is countably $\sum$-$\mathcal{L}$-injective,
by Corollary~\ref{cor:(2.4.2)}.
\end{proof}

\begin{cor}\label{cor:(2.4.5)} Let $\rho$ be any $P$-filter. Then the following conditions are equivalent.

$\left(1\right)$ $E$ is countably $\sum$-$\mathcal{L}_{(\rho,\infty)}$-injective.

$\left(2\right)$ Every ascending chain $I_{1}\subseteq I_{2}\subseteq...$
of $E$-annihilators in $R$ with $I_{\infty}=\bigcup_{j=1}^{\infty}I_{j}$
is $s$-$\mathcal{L}_{(\rho,\infty)}$-dense left ideal in $R,$ terminates.

$\left(3\right)$ The following conditions hold.

$\quad$$\left(a\right)$ Every ascending chain $I_{1}\subseteq I_{2}\subseteq...$
of $E$-annihilators in $R$ with $I_{j}$ is $s$-$\mathcal{L}_{(\rho,\infty)}$-dense
left ideals of $R,$$\forall j\in\mathbb{N},$ terminates.

$\quad$$\left(b\right)$ For every ascending chain $I_{1}\subseteq I_{2}\subseteq...$
of $E$-annihilators in $R$ with $I_{\infty}=\bigcup_{j=1}^{\infty}I_{j}$
is $s$-\textup{$\mathcal{L}_{(\rho,\infty)}$}-dense left ideal
in $R,$ there exists a positive integer $n$ such that $I_{n}$ is $s$-$\mathcal{L}_{(\rho,\infty)}$-dense in $R.$

$\left(4\right)$ $E$ is $\sum$-$\mathcal{L}_{(\rho,\infty)}$-injective.
\end{cor}
\begin{proof} By Lemma~\ref{lem:(2.1.9)}, Lemma~\ref{lem:(2.3.9)} and Theorem~\ref{thm:(2.4.4)}, we have the equivalence
of $\left(1\right),\,\left(2\right)$ and $\left(3\right).$

$\left(1\right)\Leftrightarrow\left(4\right)$ By Corollary~\ref{cor:(2.4.2)}.
\end{proof}

\begin{cor}\label{cor:(2.4.6)} \textbf{(\cite[Theorem 16.16, p. 98]{Cha06})} Let $\tau$ be any hereditary torsion theory and let $E$ be $\tau$-injective module. Then the following conditions are equivalent.

$\left(1\right)$ $E$ is countably $\sum$-$\tau$-injective.

$\left(2\right)$ Every ascending chain $I_{1}\subseteq I_{2}\subseteq...$
of $E$-annihilators in $R$ with $I_{\infty}=\bigcup_{j=1}^{\infty}I_{j}$
is  $\tau$-dense left ideal in $R,$ terminates.

$\left(3\right)$ The following conditions hold.

$\quad$$\left(a\right)$ Every ascending chain $I_{1}\subseteq I_{2}\subseteq...$
of $E$-annihilators in $R$ with $I_{j}$ is  $\tau$-dense
left ideals of $R,\,\forall j\in\mathbb{N},$ terminates.

$\quad$$\left(b\right)$ For every ascending chain $I_{1}\subseteq I_{2}\subseteq...$
of $E$-annihilators in $R$ with $I_{\infty}=\bigcup_{j=1}^{\infty}I_{j}$
is  $\tau$-dense left ideal in $R,$ there exists a positive
integer $n$ such that $I_{n}$ is  $\tau$-dense in $R.$

$\left(4\right)$ $E$ is $\sum$-$\tau$-injective.
\end{cor}
\begin{proof} By taking a $P$-filter $\rho=\rho_{\tau}$ and applying Corollary~\ref{cor:(2.4.5)}.
\end{proof}

\begin{cor}\label{cor:(2.4.7)} \textbf{(\cite[Proposition 3, p. 184]{Fai66})} The following conditions on an injective module $E$ are equivalent.

$\left(1\right)$ $E$ is countably $\sum$-injective.

$\left(2\right)$ $R$ satisfies the $ACC$ on the $E$-annihilators
left ideals.

$\left(3\right)$ $E$ is $\sum$-injective.
\end{cor}
\begin{proof} By taking $\rho=\Re$ and applying Corollary~\ref{cor:(2.4.5)}.
\end{proof}

\begin{cor}\label{cor:(2.4.8)} Let $\mathcal{L}$ satisfy  conditions $\left(\alpha\right),\,\left(\mu\right)$
and $\left(\delta\right)$ and let $\{E_{i}\mid1\leq i\leq n\}$ be a family of modules. If $E_{i}$ is $\sum$-$\mathcal{L}$-injective,
$\,\forall i=1,2,...,n,$ then $\bigoplus_{i=1}^{n}E_{i}$ is $\sum$-$\mathcal{L}$-injective.
\end{cor}
\begin{proof} Since $E_{i}$ is $\sum$-$\mathcal{L}$-injective, \,$\forall i=1,2,...,n,$  $E_{i}^{(\mathbb{N})}$ is $\mathcal{L}$-injective, $\,\forall i=1,2,...,n.$ Thus by Corollary~\ref{cor:(2.1.5)}, we have that $\bigoplus_{i=1}^{n}E_{i}^{(\mathbb{N})}$
is $\mathcal{L}$-injective. Since
$(\bigoplus_{i=1}^{n}E_{i})^{(\mathbb{N})}=(E_{1}\oplus E_{2}\oplus...\oplus E_{n})^{(\mathbb{N})}=E_{1}^{(\mathbb{N})}\oplus E_{2}^{(\mathbb{N})}\oplus...\oplus E_{n}^{(\mathbb{N})}=\bigoplus_{i=1}^{n}E_{i}^{(\mathbb{N})},$
  $(\bigoplus_{i=1}^{n}E_{i})^{(\mathbb{N})}$ is $\mathcal{L}$-injective.
Hence $\bigoplus_{i=1}^{n}E_{i}$ is $\sum$-$\mathcal{L}$-injective,
by Corollary~\ref{cor:(2.4.2)}.
\end{proof}

\begin{cor}\label{cor:(2.4.9)} Let $\rho_{1}$ and $\rho_{2}$ be any two $P$-filters and let $\{E_{i}\mid1\leq i\leq n\}$ be a family of modules.
If $E_{i}$ is $\sum$-$\mathcal{L}_{(\rho_{1},\rho_{2})}$-injective, $\forall i=1,2,...,n,$
 then $\bigoplus_{i=1}^{n}E_{i}$ is $\sum$-$\mathcal{L}_{(\rho_{1},\rho_{2})}$-injective.
\end{cor}
\begin{proof}
By Lemma~\ref{lem:(2.1.9)} and Corollary~\ref{cor:(2.4.8)}.
\end{proof}

We can now state the following result, found in \cite[p. 98]{Cha06} as a corollary.

\begin{cor}\label{cor:(2.4.10)} Let $\tau$ be any hereditary torsion theory and let $\{E_{i}\mid1\leq i\leq n\}$
be a family of modules. If $E_{i}$ is $\sum$-$\tau$-injective,\, $\forall i=1,2,...,n,$ then $\bigoplus_{i=1}^{n}E_{i}$ is $\sum$-$\tau$-injective.
\end{cor}
\begin{proof} By taking the two $P$-filters $\rho_{1}=\rho_{\tau}$ and $\rho_{2}=\Re$ and applying Corollary~\ref{cor:(2.4.9)}.
\end{proof}


\begin{thebibliography}{9999}
\bibitem{AnFu92} F. W. Anderson and K. R. Fuller, {\it Rings and Categories of Modules}, Graduate Texts in Mathematics, Springer-Verlag, New York (1992).
\bibitem{AzMbVa75} G. Azumaya,  F. Mbuntum  and K. Varadarajan,  On $M$-projective and $M$-injective modules, Pacific J. Math., 59 (1) (1975), 9-16.
\bibitem{Bae40} R. Baer, Abelian groups that are direct summands of every containing abelian group, Proc. Amer. Math. Soc., 46 (1940), 800-806.
\bibitem{Bea72} J. A. Beachy, A generalization of injectivity, Pacific J. Math., 41 (2) (1972), 313-327.
\bibitem{Bla90} P. Bland, A note on quasi-divisible modules, Comm. Algebra 18 (1990), 1953-1959.
\bibitem{Cha06} S. Charalambides,   {\it Topics in Torsion Theory}, Ph.D. Thesis, University of Otago, New Zealand (2006).
\bibitem{Cri01} S. Crivei, A note on $\tau$-quasi-injective modules, Studia Univ. "Babe\c{s} Bolyai", Mathematica, 46 (3) (2001), 33\,-39.
\bibitem{Cri04} S. Crivei,  {\it  Injective Modules Relative to Torsion Theories}, Editura Fundat IEI Pentru Studii Europene, Cluj-Napoca  (2004).
\bibitem{Dau91} J. Dauns,  Classes of modules, Forum Math., 3 (1991), 327-338.
\bibitem{DaZh06} J. Dauns and Y. Zhou,  {\it Classes of Modules}, Chapman and Hall/CRC, Boca Raton  (2006).
\bibitem{Fai66} C. Faith, Rings with ascending condition on annihilators, Nagoya Math. J., 27 (1966), 179-191.
\bibitem{Fuc69} L. Fuchs, On quasi-injective modules, Ann. Scuola Norm. Sup. Pisa, 23(1969), 541-546.
\bibitem{Gol86} J. S. Golan, {\it Torsion Theories}, Longman Scientific and Technical, New York (1986).
\bibitem{Jir75} J. Jir\'{a}sko, Generalized injectivity, Commentationes Mathematicae Universitatis Carolinae, 16 (4) (1975), 621-636.
\bibitem{Kas82} F. Kasch, {\it Modules and Rings}, Academic Press, New York (1982).
\bibitem{MoMu90} S. H. Mohamed  and B. J. M\"{u}ller, {\it Continuous and Discrete Modules}, London Math. Soc. Lecture Notes,vol. 147, Cambridge University Press, London (1990).
\bibitem{PaZh94} S. S. Page and  Y. Zhou, On direct sums of injective modules and chain conditions, Canad. J. Math., 46 (1994), 634-647.
\bibitem{Smi97} P. F. Smith, {\it Injective Modules and their Generalizations}, University of Glasgow, Dept. of Math., Preprint series, No. 97/7, Glasgow (1997).
\bibitem{Ste75} B. Stenstr\"{o}n, {\it Rings of Quotients}, Springer-Verlag, New York (1975).
\bibitem{Yi98}  Okyeon Yi, On injective modules and locally nilpotent endomorphisms of injective modules, Math. J. Okayama Univ., 40 (1998), 7-13.
\end{thebibliography}
\end{document}